\documentclass[12pt,a4paper]{amsart}
\usepackage{amssymb,amsmath,amsthm}
\usepackage{enumerate}
\usepackage{graphicx}

\theoremstyle{plain}
\newtheorem{theorem}{\bf Theorem}[section]
\newtheorem{lemma}[theorem]{\bf Lemma}

\theoremstyle{definition}
\newtheorem{definition}[theorem]{\bf Definition}

\theoremstyle{remark}
\newtheorem*{remark*}{\bf Remark}

\numberwithin{equation}{section}

\newcommand{\N}{{\mathbb N}}
\newcommand{\Z}{{\mathbb Z}}
\newcommand{\R}{{\mathbb R}}
\newcommand{\C}{{\mathbb C}}

\newcommand{\rme}{{\rm e}}
\newcommand{\rmd}{{\rm d}}

\newcommand{\cG}{{\mathcal G}}

\newcommand{\cL}{{\mathcal L}}
\newcommand{\cM}{{\mathcal M}}

\newcommand{\cR}{{\mathcal R}}

\newcommand{\sig}{\sigma}

\DeclareMathOperator{\vol}{vol}

\DeclareMathOperator{\Dom}{Dom}
\DeclareMathOperator{\Spec}{Spec}
\DeclareMathOperator{\Ker}{Ker} 
\DeclareMathOperator{\Tr}{Tr} 

\renewcommand{\Re}{\hbox{{\rm Re}}\,}
\renewcommand{\Im}{\hbox{{\rm Im}}\,}
\DeclareMathOperator{\supp}{supp}
\newcommand{\hash}{\#}

\newcommand{\abs}[1]{{\lvert#1\rvert}}

\newcommand{\wt}\widetilde

\newcommand{\ess}{{\text{\rm ess}}}
\newcommand{\loc}{{\text{\rm loc}}}
\newcommand{\inte}{{\text{\rm int}}}
\newcommand{\exte}{{\text{\rm ext}}}
\newcommand{\even}{{\text{\rm even}}}
\newcommand{\odd}{{\text{\rm odd}}}
\newcommand{\length}{\rho}
\newcommand{\vf}{\varphi_v^\varkappa}

\title[Non-Weyl Resonance Asymptotics for Quantum Graphs]
{Non-Weyl Resonance Asymptotics\\ for Quantum Graphs}

\author[Davies]{E.~B.~Davies}
\address[Davies]{Department of Mathematics,
King's College London, Strand, London WC2R  2LS, U.K.}
\email{E.Brian.Davies@kcl.ac.uk}
\author[Pushnitski]{A.~Pushnitski}
\address[Pushnitski]{Department of Mathematics,
King's College London, Strand, London WC2R  2LS, U.K.}
\email{alexander.pushnitski@kcl.ac.uk}

\date{1 March 2010}
\begin{document}



\begin{abstract}
We consider the resonances of a quantum graph $\cG$ that consists of
a compact part with one or more infinite leads attached to it. We
discuss the  leading term of the asymptotics of the number of
resonances of $\cG$ in a disc of a large radius. We call $\cG$  a
\emph{Weyl graph} if the  coefficient in front of this leading term
coincides with the volume of the compact part of $\cG$. We give an
explicit  topological criterion for a graph to be Weyl. In the final
section we analyze a particular example in some detail to explain
how the transition from the Weyl to the non-Weyl case occurs.
\end{abstract}

\maketitle

\section{Introduction}\label{intro}

\subsection{Quantum graphs}
Let $\cG_0$ be a finite compact metric graph. That is, $\cG_0$ has
finitely many edges and each edge is equipped with coordinates
(denoted $x$) that identify this edge with a bounded interval of the
real line. We choose some subset of  vertices of $\cG_0$, to be
called \emph{external vertices}, and attach one or more copies of
$[0,\infty)$, to be called \emph{leads}, to each external vertex;
the point $0$ in a lead is thus identified with the relevant
external vertex. We call the thus extended graph $\cG$. We assume
that $\cG$ has no ``tadpoles'', i.e.\ no edge starts and ends at the
same vertex; this can always be achieved by introducing additional
vertices, if necessary. In order to distinguish the edges of $\cG_0$
from the leads, we will call the former the \emph{internal edges} of
$\cG$.

In $L^2(\cG)$ we consider the self-adjoint operator
$H=-\frac{d^2}{dx^2}$ with the continuity condition and the
Kirchhoff boundary condition at each vertex of $\cG$; see
Section~\ref{sec.b} for the precise definitions. The metric graph
$\cG$ equipped with the self-adjoint operator $H$ in $L^2(\cG)$ is
called the \emph{quantum graph}. We refer to the surveys
\cite{kuch1,kuch} for a general exposition of quantum graph theory.

If the set of leads is non-empty, it  is easy to show by standard
techniques (see e.g.\ \cite[Lemma 1]{Ong}) that the spectrum of $H$
is $[0,\infty)$. The operator $H$ may have embedded eigenvalues.

\subsection{Resonances of $H$}
The ``classical'' definition of resonances is
\begin{definition}\label{def0}
We will say that $k\in\C$, $k\not=0$, is a \emph{resonance} of $H$
(or, by a slight abuse of terminology, a resonance of $\cG$) if
there exists a \emph{resonance eigenfunction} $f\in L_\loc^2(\cG)$,
$f\not\equiv0$, which satisfies the equation
\begin{equation}
-f''(x)=k^2 f(x), \quad x\in \cG, \label{a0a}
\end{equation}
on each edge and lead of $\cG$, is continuous on $\cG$, satisfies
the Kirchhoff's boundary condition at each vertex of $\cG$ and the
\emph{radiation condition} $f(x)=f(0)e^{ikx}$ on each lead of $\cG$.
We denote the set of all resonances of $H$ by $\cR$.
\end{definition}

Any solution to \eqref{a0a} on a lead $\ell=[0,\infty)$ satisfies
$f(x)=\gamma_\ell e^{ikx}+\gamma'_\ell e^{-ikx}$; the above
definition  requires that there exists a non-zero solution with all
coefficients $\gamma'_\ell$ vanishing. It is easy to see that all
resonances must satisfy $\Im k\leq 0$; 
indeed, if $k_0$ with $\Im k_0>0$ is a resonance then the 
corresponding resonance eigenfunction is in $L^2(\cG)$, so 
$k_0^2$ is an eigenvalue of $H$, which is impossible
since $k_0^2\notin[0,\infty)$.
As we will only be
interested in the asymptotics of the number of resonances in large
disks, we exclude the case $k=0$ from further consideration. In the
absence of leads, the spectrum of $H$ consists of non-negative
eigenvalues and $k\not=0$ is a resonance if and only if $k\in\R$ and
$k^2$ is an eigenvalue of $H$.

It is well known (see e.g.\ \cite{EL1,EL2}) that the above
``classical'' definition of a resonance coincides with the
definition via exterior complex scaling (see \cite{AC,Simon,SZ}). In
the complex scaling approach, the resonances of $H$ are identified
with the eigenvalues of an auxiliary non-selfadjoint operator
$H(i\theta)$, $\theta\in(0,\pi)$. The \emph{algebraic multiplicity}
of a resonance is then defined as the algebraic multiplicity of the
corresponding eigenvalue of $H(i\theta)$. We discuss this in more
detail in Section~\ref{sec.b}, where we show that the multiplicity is independent of $\theta$. In particular, we show (in
Theorem~\ref{thm.b3}) that any $k\in\R$, $k\not=0$, is a resonance
if and only if $k^2$ is an eigenvalue of $H$ and in this case the
corresponding multiplicities coincide.

We define the \emph{resonance counting function} by
$$
N(R)=\hash\{k: k\in\cR,\quad \abs{k}\leq R\}, \quad R>0,
$$
with the convention that each resonance is counted with its
algebraic multiplicity taken into account. Note that the set $\cR$
of resonances is invariant under the symmetry $k\to -\overline{k}$,
so this method of counting yields, roughly speaking, twice as many
resonances as one would obtain if one imposed an additional
condition $\Re (k)\geq 0$. In particular, in the absence of leads,
$N(R)$ equals twice the number of eigenvalues $\lambda\not=0$ of $H$
(counting multiplicities) with $\lambda\leq R^2$.

\subsection{Main result}

This paper is concerned with the asymptotics of the resonance
counting function $N(R)$ as $R\to\infty$. We say that $\cG$ is a
\emph{Weyl graph}, if
\begin{equation}
N(R)=\frac2\pi \vol(\cG_0) R+o(R), \quad \mbox{ as $R\to\infty$,}
\label{weyl}
\end{equation}
where $\vol (\cG_0)$ is the sum of the lengths of the edges of
$\cG_0$. If  there are no leads then $H$ has pure point spectrum,
resonances are identified with eigenvalues of $H$ and Weyl's law
\eqref{weyl} may be proved by Dirichlet-Neumann bracketing. Thus,
every compact quantum graph is Weyl in our sense. As we show below,
in the presence of leads this may not be the case.

We call an external vertex $v$ of $\cG$ \emph{balanced} if the
number of leads attached to $v$ equals the number of internal edges
attached to $v$. If  $v$ is not balanced, we call it
\emph{unbalanced}. Our main result is

\begin{theorem}\label{th1}
One has
\begin{equation}
N(R)=\frac{2}{\pi}WR+O(1), \quad \mbox{as $R\to\infty$,} \label{a4}
\end{equation}
where the coefficient $W$ satisfies $0\leq W\leq \vol(\cG_0)$. One
has $W=\vol(\cG_0)$ if and only if every external vertex of $\cG$ is
unbalanced.
\end{theorem}

This theorem shows, in particular, that as the graph becomes larger
and more complex the failure of Weyl's law becomes increasingly
likely in an obvious sense.

\subsection{Discussion}
The simplest example of a graph $\cG$ with a balanced external
vertex occurs when exactly one lead $\ell$ and exactly one internal
edge $e$ meet at a vertex. In this case, one can merge $e$ and
$\ell$ into a new lead; this will not affect the resonances of $\cG$
but will reduce $\vol \cG_0$. This already shows that $\cG$ cannot
be Weyl. Section~\ref{sec.f} discusses the second most simple
example.

Our proof of Theorem~\ref{th1} consists of two steps. The first step
is to identify the set $\cR$  of resonances with the set of zeros of
$\det A(k)$, where $A(k)$ is a certain analytic matrix-valued
function. This identification is straightforward, but it has a
subtle aspect: this is to show  that the algebraic multiplicity of a
resonance coincides with the order of the zero of $\det A(k)$. This
is done in Sections~\ref{sec.c}--\ref{sec.d} by employing a range of
rather standard techniques of spectral theory, including a resolvent
identity which involves the Dirichlet-to-Neumann map.

The function $\det A(k)$ turns out to be an exponential polynomial.
By a classical result (Theorem~\ref{langer}), the asymptotics of the
zeros of an exponential polynomial can be explicitly expressed in
terms of the coefficients of this polynomial. Thus, the second step
of our proof is a direct and completely elementary analysis of the
matrix $A(k)$ which allows us to relate the required information
about the coefficients of the polynomial $\det A(k)$ to the question
of whether the external vertices of $\cG$ are balanced. This is done
in Section~\ref{sec.e}.

Resonance asymptotics of Weyl type have been established for
compactly supported potentials on the real line, a class of
super-exponentially decaying potentials on the real line, compactly
supported potentials on cylinders and Laplace operators on surfaces
with finite volume hyperbolic cusps in \cite{zwor,froe,chri,parn}
respectively. The proofs rely upon theorems about the zeros of
certain classes of entire functions. Likewise, our analysis uses a
simple classical result (Theorem~\ref{langer}) about zeros of
exponential polynomials.

The situation with resonance asymptotics for potential and obstacle
scattering in Euclidean space in dimensions greater than one and in
hyperbolic space is more complicated and still not fully understood;
the current state of knowledge is described in \cite{stef,borth}.
Here we remark only that generically, the resonance asymptotics in
the multi-dimensional case is not given by the Weyl formula. We hope
that Theorem~\ref{th1} can provide some insight to the
multi-dimensional case.

Resonances for quantum graphs have been discussed in a recent
publication \cite{EL2}. Our paper has very little technical content
in common with \cite{EL2}, in spite of their common theme.

\subsection{Example}
In Section~\ref{sec.f} we consider the resonances of a particularly
simple quantum graph which can be described as a circle with two
leads attached to it. Theorem~\ref{th1} says that if  the leads are
attached at different points on the circle, the corresponding
quantum graph is Weyl, and if they are attached at the same point,
we have a non-Weyl graph. When the two points where the leads are
attached move closer to each other and eventually coalesce, one
observes the transition from the Weyl to the non-Weyl case. We study
this transition in much detail. We show that as the two external
vertices get closer, ``half'' of the resonances move off to
infinity. In the course of this analysis, we also obtain bounds on
the positions of individual resonances for this model.

The same example was recently considered by Exner and Lipovsky
\cite{EL2} subject to general boundary conditions that include the
Kirchhoff's boundary condition case as a singular limit. Although
some of their results are broadly similar to ours, none of our
theorems  may be found in \cite{EL2}.

\section{Resonances via complex scaling}\label{sec.b}

Here we introduce the necessary notation, recall the definition of
resonances via the complex scaling procedure and show that  the
resonances on the real axis coincide with the eigenvalues of $H$.

\subsection{Notation}\label{notation}
Let $E^{\inte}$ be the set of all internal edges of $\cG$ (i.e.\ the
set of all edges of $\cG_0$) and let $E^{\exte}$ be the set of all
leads; we also denote $E=E^\inte\cup E^\exte$. The term ``edge''
without an adjective will refer to any element of $E$. For $e\in
E^\inte$, we denote by $\length(e)$ the length of $e$; i.e.\ an edge
$e\in E^\inte$ is identified with the interval $[0,\length(e)]$.

Let $V$ be the set of all vertices of $\cG$, let $V^\exte$ be the
set of all external vertices, and let $V^\inte=V\setminus V^\exte$;
the elements of $V^\inte$ will be called \emph{internal vertices}.
The degree of a vertex $v$ is denoted by $d(v)$. The number of leads
attached to an external vertex $v$ is denoted by $q(v)$; we also set
$q(v)=0$ for $v\in V^\inte$.

If an edge or a lead $e$ is attached to a vertex $v$, we write $v\in
e$. If two vertices $u,v$ are connected by one or more edges, we
write $u\sim v$.

We denote by $\cG_\infty$ the graph $\cG$ with all the internal
edges and vertices removed. We let $\chi_0$ and $\chi_\infty$ be the
characteristic functions of $\cG_0$ and $\cG_\infty$.

Let $f:\cG\to\C$ be a function such that the restriction of $f$ onto
every edge is continuously differentiable. Then for $v\in V$, we
denote by $N_v f$ the sum of the outgoing derivatives of $f$ at $v$
over all edges attached to $v$. If $v$ is an external vertex, we
denote by $N_v^\inte f$ (resp. $N_v^\exte f$) the sum of all
outgoing derivatives of $f$ at $v$ over all internal edges (resp.
leads) attached to $v$.

Let $\wt C(\cG)$ be the class of functions $f:\cG\to\C$ which are
continuous on $\cG\setminus V^\exte$ and such that for each external
vertex $v$ the function $f(x)$ approaches a limiting value (to be
denoted by $D_v^\inte f$) as $x$ approaches $v$ along any internal
edge and $f(x)$ approaches another limiting value (to be denoted by
$D_v^\exte f$) as $x$ approaches $v$ along any lead.

For any finite set $A$, we denote by $\abs{A}$ the number of
elements of $A$. We will use the identity
\begin{equation}
\sum_{v\in V} d(v)=2\abs{E^\inte}+\abs{E^\exte}. \label{b1}
\end{equation}
Finally, we use the notation $\C_+=\{z\in\C:  \Im z>0\}$.

\subsection{The operator $H(\varkappa)$\label{sec2.2}}
The domain of the self-adjoint operator $H$ consists of all
continuous functions $f:\cG\to\C$ such that the restriction of $f$
onto any $e\in E$ lies in the Sobolev space $W^2_2(e)$, and $f$
satisfies the Kirchhoff boundary condition $N_v f=0$ on every vertex
$v$ of $\cG$.

For $\varkappa\in\R$, let $U(\varkappa):L^2(\cG)\to L^2(\cG)$ be the
unitary operator which acts as identity on $L^2(\cG_0)$ and as a
dilation on all leads $\ell=[0,\infty)$:
\begin{equation}
(U(\varkappa)f)(x)=e^{\varkappa/2}f(e^\varkappa x), \quad x\in\ell.
\label{a1}
\end{equation}
Note that we have $U(\varkappa)^*=U(-\varkappa)$ for any
$\varkappa\in\R$. Consider the operator
\begin{equation}
H(\varkappa)=U(\varkappa)HU(-\varkappa). \label{a2}
\end{equation}
This operator admits an analytic continuation to $\varkappa\in\C$,
which we describe below.

\begin{definition}
For $\varkappa\in\C$, the operator $H(\varkappa)$ in $L^2(\cG)$ acts
according to the formula
\begin{equation}
(H(\varkappa)f)(x) =
\begin{cases}
-f''(x), &\text{if $x\in\cG_0$},
\\
-e^{-2\varkappa}f''(x), &\text{if $x\in\cG_\infty$.}
\end{cases}
\label{b2}
\end{equation}
The domain of $H(\varkappa)$ is defined to be the set of all
$f:\cG\to\C$ which satisfy the following conditions:

(i) the restriction of $f$ onto any $e\in E$ lies in the Sobolev
space $W^2_2(e)$;

(ii) $f\in\wt C(\cG)$;

(iii) $f$ satisfies the condition $N_v f=0$  at every internal
vertex $v$;

(iv) For any $v\in V^\exte$, one has
\begin{gather}
D_v^\inte f-e^{-\varkappa/2}D_v^\exte f=0; \label{b3}
\\
N_v^\inte f+ e^{-3\varkappa/2} N_v^\exte f=0. \label{b4}
\end{gather}
\end{definition}
In particular, $H(0)$ is the operator called $H$ so far. For complex
$\varkappa$, the operator $H(\varkappa)$ is in general
non-selfadjoint. A standard straightforward computation shows that
for any $\varkappa\in\C$ the operator $H(\varkappa)$ is closed and
\begin{equation}
H(\varkappa)^*=H(\overline{\varkappa}). \label{b5}
\end{equation}

\subsection{Resonances via complex scaling}

The following theorem is standard in the method of complex scaling;
see \cite{AC,Simon,SZ,EL1}:
\begin{theorem}\label{thm.b2}
The family of operators $H(\varkappa)$, $\varkappa\in\C$, is
analytic in the sense of Kato (see e.g.\ \cite[Section
XII.2]{reedsimon}) and the identity
\begin{equation}
H(\varkappa+\varkappa_0)=U(\varkappa_0)H(\varkappa)U(-\varkappa_0),
\quad \forall \varkappa\in\C, \  \forall \varkappa_0\in\R \label{a3}
\end{equation}
holds. The essential spectrum of $H(\varkappa)$ coincides with the
half-line $e^{-2\varkappa}[0,\infty)$. Let $\theta\in(0,\pi)$; then
the sector  $0<\arg\lambda<2\pi-2\theta$, $\lambda\not=0$, contains
no eigenvalues of $H(i\theta)$, and any $\lambda\not=0$ in the
sector $2\pi-2\theta<\arg\lambda\leq2\pi$ is an eigenvalue of
$H(i\theta)$ if and only if $\lambda=k^2$ with $k\in\cR$.
\end{theorem}
For completeness, we give the proof in Section~\ref{sec.d}.

As  $\theta\in(0,\pi)$ increases monotonically,  the essential
spectrum $e^{-2i\theta}[0,\infty)$  of $H(i\theta)$ rotates
clockwise, uncovering more and more  eigenvalues $\lambda$. These
eigenvalues are identified with the resonances $k$ of $H$ via
$\lambda=k^2$. If $\lambda\not=0$ is an eigenvalue of $H(i\theta)$,
$\theta\in(0,\pi)$, $2\pi-2\theta<\arg\lambda\leq 2\pi$, Kato's
theory of analytic perturbations implies that the eigenvalue and
associated Riesz spectral projection depend analytically on
$\theta$. Noting (\ref{a3}) and using analytic continuation it
follows that the algebraic multiplicity of $\lambda$ is independent
of $\theta$. It is easy to see directly that the geometric
multiplicity of $\lambda$ is also independent of $\theta$. The
\emph{algebraic  (resp. geometric) multiplicity} of a resonance $k$
is defined as the algebraic (resp. geometric) multiplicity of the
corresponding eigenvalue $\lambda=k^2$ of $H(i\theta)$.

\subsection{Resonances on the real line}

The geometric multiplicities of resonances will not play any role in
our analysis. However, we note that for the Schr\"odinger operator
on the real line, resonances can have arbitrary large algebraic
multiplicity \cite{Kor}, while their geometric multiplicity is
always equal to one. This gives an example of resonances with
distinct algebraic and geometric multiplicities. It would be
interesting to see if one can have distinct algebraic and geometric
multiplicities of resonances for quantum graphs in the situation we
are discussing. We have nothing to say about this except for the
case of the resonances on the real line:

\begin{theorem}\label{thm.b3}
\begin{enumerate}[\rm (i)]
\item
If $k\in\R$, $k\not=0$, is a resonance of $H$ then the algebraic and
geometric multiplicities of $k$ coincide.
\item
Any $k\in\R$, $k\not=0$, is a resonance of $H$ if and only if $k^2$
is an eigenvalue of $H$ and the multiplicity of the resonance $k$
coincides with the multiplicity of the eigenvalue $k^2$.
\end{enumerate}
\end{theorem}

\begin{proof}
1. Let $\lambda>0$ be an eigenvalue of $H$ with the eigenfunction
$f$. If $\ell=[0,\infty)$ is a lead, then $f(x)=\gamma_\ell
e^{ikx}+\gamma_\ell' e^{-ikx}$, $x\in \ell$, where $k^2=\lambda$.
Since $f\in L^2(\ell)$, we conclude that
$\gamma_\ell=\gamma_\ell'=0$ and so $f\equiv0$ on all leads. It
follows that $f\in\Dom H(i\theta)$ for all $\theta$ and
$H(i\theta)f=\lambda f$. This argument proves that
\begin{equation}
\dim \Ker (H(i\theta)-\lambda I) \geq \dim \Ker(H-\lambda I).
\label{b7}
\end{equation}

2. Let $f\in\Ker (H(i\theta)-\lambda I)$, $\lambda>0$,
$\theta\in(0,\pi)$. Let us prove that $f$ vanishes identically on
all leads. Let $\lambda=k^2$, $k>0$. On every lead, we have
\begin{equation}
f(x)=f(0)\exp(ie^{i\theta}kx). \label{b8}
\end{equation}
Consider the difference
\begin{equation}
\omega(f) = \int_{\cG_0} \abs{f'(x)}^2 dx -
\lambda\int_{\cG_0}\abs{f(x)}^2 dx = \int_{\cG_0} \abs{f'(x)}^2 dx +
\int_{\cG_0}f''(x) \overline{f(x)} dx. \label{b9}
\end{equation}
Integrating by parts, we get
$$
\omega(f) = -\sum_{v\in V^\exte}(N_v^\inte f) \overline{D_v^\inte
f}.
$$
Using the boundary condition \eqref{b3} and formula \eqref{b8}, we
obtain
$$
\omega(f) = ik\sum_{v\in V^\exte} \abs{D_v^\exte f}^2 q(v).
$$
By the definition \eqref{b9} of $\omega(f)$, we have $\Im
\omega(f)=0$. This yields that $\abs{D_v^\exte f}=0$ on all external
vertices $v$. By \eqref{b8}, it follows that $f$ vanishes
identically on all leads.

3. By combining the previous step of the proof with \eqref{b3} and
\eqref{b4} we obtain $D_v^\inte f=N_v^\inte f=0$. It follows that
for any $f\in\Ker (H(i\theta)-\lambda I)$, $\lambda>0$,
$\theta\in(0,\pi)$, we have $f\in\Dom H$ and $Hf=\lambda f$. This
argument also proves that
\begin{equation}
\dim \Ker (H-\lambda I)\geq \dim \Ker (H(i\theta)-\lambda I).
\label{b10}
\end{equation}

4. It remains to prove that if $\lambda>0$ is an eigenvalue of
$H(i\theta)$, $\theta\in(0,\pi)$, then its algebraic  and geometric
multiplicities   coincide. Suppose this is not the case. Then there
exist non-zero elements $f,g\in\Dom H(i\theta)$ such that
$H(i\theta)g=\lambda g$ and $(H(i\theta)-\lambda I)f=g$.

By step 2 of the proof, $g$ vanishes on all leads. It follows that
on all leads the function $f$ satisfies \eqref{b8}. Next, since
$g(x)=-f''(x)-\lambda f(x)$ on $\cG_0$, we have
\begin{equation}
0 < \int_{\cG_0}\abs{g(x)}^2dx = -\int_{\cG_0}(f''(x)+\lambda
f(x))\overline{g(x)}dx. \label{b11}
\end{equation}
Integrating by parts, we get
\begin{multline}
-\int_{\cG_0}(f''(x)+\lambda f(x))\overline{g(x)}dx = -
\int_{\cG_0}f(x)(\overline{g''(x)}+\lambda\overline{g(x)})dx
\\
+ \sum_{v\in V^\exte}(N_v^\inte f)(\overline{D_v^\inte g}) -
\sum_{v\in V^\exte}(D_v^\inte f)(\overline{N_v^\inte g}).
\label{b12}
\end{multline}
Consider the three terms in the r.h.s.\ of \eqref{b12}. The first
term vanishes since $H(i\theta)g=\lambda g$. Next, since $g\equiv 0$
on $\cG_\infty$, we have $D_v^\exte g=N_v^\exte g=0$ for any $v\in
V^\exte$. By the boundary conditions \eqref{b3} and \eqref{b4} for
$g$ it follows that $D_v^\inte g=N_v^\inte g=0$. Thus, the second
and third terms in the r.h.s.\ of \eqref{b12} also vanish. This
contradicts \eqref{b11}.
\end{proof}

\section{Proof of Theorem~\ref{th1}}\label{sec.e}
Here we describe the resonances as zeros of  $\det A(k)$, where
$A(k)$ is certain entire matrix-valued function. Using this
characterisation, we prove our main result.

\subsection{Definition of $A(k)$}\label{sec.d1}

Fix $k\in\C_+$. Let $\cL(k)$ denote the space of all solutions $f\in
L^2(\cG)$ to $-f''=k^2f$ on $\cG$ without any boundary conditions.
The restriction of $f\in\cL(k)$ onto any internal edge $e$ has the
form $f_e(x)=\alpha_e e^{ikx}+\beta_e e^{-ikx}$, and the restriction
of $f$ onto any lead $\ell$ has the form $f_\ell(x)=\gamma_\ell
e^{ikx}$. Thus, $\dim \cL(k)=2\abs{E^\inte}+\abs{E^\exte}$.

Let us describe in detail the set of all conditions on $f\in\cL(k)$
required to ensure that $f$ is a resonance eigenfunction. If $f_e$
denotes the restriction of $f$ to an edge $e$, then we can write the
continuity conditions at the vertex $v$ as
\begin{equation}
f_e(v)=\zeta_v, \quad \forall e\ni v, \label{d1}
\end{equation}
where $\zeta_v\in\C$ is an auxiliary variable. We also have the
condition
\begin{equation}
N_v f=0, \quad v\in V. \quad \label{d2}
\end{equation}
Writing down conditions \eqref{d1}, \eqref{d2} for every vertex
$v\in V$, we obtain
$$
N = \sum_{v\in V}d(v)+\abs{V} = 2\abs{E^\inte}+\abs{E^\exte}+\abs{V}
$$
conditions. Our variables are $\zeta_v$, $\alpha_e$, $\beta_e$,
$\gamma_\ell$; altogether we have
$$
\abs{V}+\dim\cL(k) = \abs{V}+2\abs{E^\inte}+\abs{E^\exte} = N
$$
variables. Let $\zeta$, $\alpha$, $\beta$, $\gamma$ be the sequences
of coordinates $\zeta_v$, $\alpha_e$, $\beta_e$, $\gamma_\ell$ of
length $\abs{V}$, $\abs{E^\inte}$, $\abs{E^\inte}$, $\abs{E^\exte}$
respectively, and let $\nu=(\zeta,\alpha,\beta,\gamma)^\top\in\C^N$.
We may write the constraints \eqref{d1}, \eqref{d2} in the form
$A\nu=0$, where $A$ is an $N\times N$ matrix. Each row of $A$
relates to one of the constraints, and each constraint is of the
form
\begin{equation}
y\cdot \zeta+a\cdot \alpha+b\cdot \beta+g\cdot\gamma=0. \label{d2b}
\end{equation}
If the constraint is of the form \eqref{d2}, then $y=0$ and $a,b,g$
all contain a multiplicative factor $ik$ which we eliminate before
proceeding. The coefficient $a_e$ is $0$, $\pm1$, or $\pm
e^{ik\length(e)}$, and the coefficient $b_e$ is $0$, $\pm1$, or $\pm
e^{-ik\length(e)}$. The coefficient $g_\ell$ is $0$ or $1$, and the
coefficient $y_v$ is $0$ or $-1$.

We have not specified the order of the rows or columns of $A(k)$.
However, the object  of importance in the sequel is the set of zeros
of $\det A(k)$, and the choice of the order of rows or columns of
$A(k)$ will not affect this set.

\subsection{Example}\label{sec.d2}
As an example, let us display the matrix $A(k)$ for a graph which
consists of two vertices $v_1$ and $v_2$, two edges $e_1$ and $e_2$
of length $\length_1$ and $\length_2$ and a lead attached at $v_1$.
In this case we have, denoting $z_j=e^{ik\length_j}$:
\begin{equation}
A(k)= \left(\begin{array}{ccccccc} 0 & 0 & z_1 & z_2 & -z_1^{-1} &
-z_2^{-1} & 0
\\0 & 0 & 1 & 1 & -1 & -1 & 1
\\-1 & 0 & 0 & 0 & 0 & 0 & 1
\\-1 & 0 & 1 & 0 & 1 & 0 & 0
\\-1 & 0 & 0 & 1 & 0 & 1 & 0
\\0 & -1 & z_1 & 0 & z_1^{-1} & 0 & 0
\\0 & -1 & 0 & z_2 & 0 & z_2^{-1} & 0
\end{array}\right).
\label{d2a}
\end{equation}

\subsection{Resonances as zeros of $\det A(k)$ }
Although $A(k)$ was defined above for $k\in\C_+$, we see that all
elements of $A(k)$ are entire functions of $k\in\C$. Thus, we will
consider $A(k)$ as an entire matrix-valued function of $k$.

In Sections~\ref{sec.c}--\ref{sec.d} we prove
\begin{theorem}\label{thm.e1}
Any $k_0\not=0$ is a resonance of $H$ if and only if $\det
A(k_0)=0$. In this case, the algebraic multiplicity of the resonance
$k_0$ coincides with the order of $k_0$ as a zero of $\det A(k)$.
\end{theorem}

The first part of this theorem is obvious: by the construction of
the matrix $A$, we have $\det A(k_0)=0$ iff there exists a non-zero
resonance eigenfunction $f\in\cL(k_0)$. The part concerning
multiplicity is less obvious. Unfortunately, we were not able to
find a completely elementary proof of this part. The proof we give
in Sections~\ref{sec.c}--\ref{sec.d} involves a standard set of
techniques of spectral theory of quantum graphs: a resolvent
identity involving the Dirichlet-to-Neumann map and a certain trace
formula.

By Theorem~\ref{thm.e1}, the question reduces to counting the total
multiplicity of zeros of the entire function $\det A(k)$ in large
discs. As it is clear from the structure of the matrix $A(k)$, its
determinant is an exponential polynomial, i.e.\ a linear combination
of the terms of the type $e^{i\sigma k}$, $\sigma\in\R$. Thus, we
need to discuss the zeros of exponential polynomials.

\subsection{Zeros of exponential polynomials}\label{sec.e1}
Exponential polynomials are entire functions $F(k)$, $k\in\C$, of
the form
\begin{equation}
F(k)=\sum_{r=1}^n a_r \rme^{i\sigma_r k}, \label{e1}
\end{equation}
where $a_r,\, \sigma_r\in\C$ are constants. The study of the zeros
of such polynomials has a long history; see e.g.\ \cite{langer} and
references therein. For more recent literature see \cite{moreno}.
Some of these results were rediscovered in
\cite{davies,davinc,incani}, where they were used to analyze the
spectra of non-self-adjoint systems of ODEs and directed finite
graphs. The asymptotic distribution of the zeros of $F$ depends
heavily on the location of the extreme points of the convex hull of
the set  $\cup_{r=1}^n\{\sigma_r\}$.

We are only interested in the case in which $\sigma_r$ are distinct
real  numbers. We denote $\sigma^-=\min\{\sigma_1,\dots,\sigma_n\}$
and $\sigma^+=\max\{\sigma_1,\dots,\sigma_n\}$. For $R>0$ we denote
by $N(R;F)$ the number of zeros of $F$ (counting the orders) in the
disc $\{k\in\C: \abs{k}<R\}$. The following classical statement is
from \cite[Theorem~3]{langer}.
\begin{theorem} \label{langer}
Let $F$ be a function of the form \eqref{e1}, where $a_r$ are
non-zero complex numbers and $\sigma_r$ are distinct real numbers.
Then there exists a constant $K<\infty$ such that all the zeros of
$F$ lie within a strip of the form $\{ k: |\Im(z)| \leq K\}$. The
counting  function $N(R;F)$ satisfies
$$
N(R;F)=\frac{\sig^+-\sig^-}{\pi}R+O(1)
$$
as $R\to +\infty$.
\end{theorem}

\subsection{Estimate for $N(R;F)$}\label{sec.e2}
Here we prove the first part of the main Theorem~\ref{th1}. Let
$F(k)=\det A(k)$. From the structure of $A(k)$ it is clear that
$F(k)$ is given by \eqref{e1} where $a_r, \sigma_r$ are real
coefficients. By Theorem~\ref{langer}, it suffices to prove that in
the representation \eqref{e1} we have
\begin{equation}
\sigma^+\leq \vol(\cG_0), \quad \sigma^-\geq -\vol(\cG_0).
\label{e2}
\end{equation}
In order to prove \eqref{e2}, let us discuss the entries of $A(k)$
in detail. For simplicity of notation we will not draw attention in
our equations to the fact that all of the matrices below depend on
$k$.

The matrix $A$ has some constant terms and some terms that are
exponential in $k$. The term $e^{ik\length(e)}$ can only appear in
the column associated with the variable $\alpha_e$ and the term
$e^{-ik\length(e)}$ can only appear in the column associated with
the variable $\beta_e$. The columns associated with the variables
$\zeta$ and $\gamma$ contain only constant terms. Since the
determinant is formed from the products of entries of $A$ where
every column contributes one entry to each product, we see that the
maximum possible value for the coefficient $\sigma_r$ in \eqref{e1}
is attained when every column corresponding to the variable
$\alpha_e$ contributes the term $e^{ik\length(e)}$ and every column
corresponding to $\beta_e$ contributes a constant term to the
product. The maximal value of $\sigma_r$ thus attained will be
exactly $\sum_{e\in E^\inte}\length(e)=\vol\cG_0$. This proves the
first inequality in \eqref{e2}. The second one is proven in the same
way by considering the minimal possible value for $\sigma_r$.

Of course,  the coefficients $a^\pm$ of the terms $e^{\pm
ik\vol(\cG_0)}$ in the representation \eqref{e1} for $\det A$ may
well happen to be zero. Theorem~\ref{th1} will be proven if we show
that these coefficients do not vanish if and only if every external
vertex of $\cG$ is unbalanced. In what follows, for an exponential
polynomial $F$ with the representation \eqref{e1} we denote by
$a^\pm(F)$ the coefficient $a_r$ of the term $e^{i\sigma_rk}$,
$\sigma_r=\pm\vol(\cG_0)$.

\subsection{Invariance of resonances with respect to a change of orientation}\label{sec.e2a}
Before proceeding with the proof, we need to discuss a minor
technical point. Our definition of the matrix $A(k)$ assumes that a
certain orientation of all internal edges of $\cG$ is fixed. Suppose
we have changed the parameterisation of an internal edge $e$ by
reversing its orientation. In other words, suppose that instead of
the variable $x\in[0,\length(e)]$ we decided to use the variable
$x'=\length(e)-x$. We claim that this change will not affect the
zeros of $\det A(k)$.

Indeed, let $A'(k)$ be the matrix corresponding to the new
parametrization. The matrix $A'(k)$ corresponds to the
parametrization of solutions $f\in\cL(k)$ on $e$ by $f(x)=\alpha'_e
e^{ikx'}+\beta'_e e^{-ikx'}$ instead of $\alpha_e e^{ikx}+\beta_e
e^{-ikx}$. We have
$$
\begin{pmatrix}
\alpha'_e \\ \beta'_e
\end{pmatrix}
=
\begin{pmatrix}
0 & e^{-ik\length(e)}
\\
e^{ik\length(e)}& 0
\end{pmatrix}
\begin{pmatrix}
\alpha_e \\ \beta_e
\end{pmatrix},
\quad \det
\begin{pmatrix}
0 & e^{-ik\length(e)}
\\
e^{ik\length(e)}& 0
\end{pmatrix}
=-1,
$$
and thus $\det A'(k)=-\det A(k)$.

\subsection{Proof of Theorem~\ref{th1}: the balanced case}\label{sec.e3}
Assume that a particular external vertex $v$ of $\cG$ is balanced.
Below we prove that the coefficient $a^+(\det A)$  vanishes.

Let us re-order the rows and columns of $A$ by reference to the
vertex $v$. We assume that $q$ internal edges and $q$ leads are
attached  to $v$, $q\geq 2$. (The case $q=1$ is trivial because one
may then merge the lead with the edge to which it is connected.)
Using the observation of Section~\ref{sec.e2a}, we can choose an
orientation of these internal edges so that they all end at $v$
(i.e.\ $v$ is identified with the point $\rho(e)$ of the intervals
$[0,\rho(e)]$). Let the first $2q$ rows of $A$ be those relating to
the conditions \eqref{d1} for the vertex $v$ and let the $(2q+1)$st
row be the one relating to the condition \eqref{d2} for the vertex
$v$. The ordering of the remaining rows does not matter. Let the
first $2q$ columns be related to the variables
$\gamma_1,\dots,\gamma_q,\alpha_1,\dots,\alpha_q$ and let the
$(2q+1)$st column be related to the variable $\zeta_v$; these variables were all defined in Section~\ref{notation}. The ordering
of the remaining columns does not matter.

We write $A$ in the block form
\begin{equation}
A=
\begin{pmatrix}
B&C
\\
D&E
\end{pmatrix}
\label{e3}
\end{equation}
where $B$ is a $(2q+1)\times(2q+1)$ matrix. For example, in the case
$q=2$ we have
\begin{equation}
B= \left(\begin{array}{ccccc}
1 & 0 & 0 & 0 & -1 \\
0 & 1 & 0 & 0 & -1 \\
0 & 0 & z_1 & 0 & -1 \\
0 & 0 & 0 & z_2 & -1 \\
1 & 1 & -z_1 & -z_2 & 0
\end{array}\right),
\label{e4}
\end{equation}
where $z_r=e^{ik\length(e_r)}$.

The determinant is the sum of the products of entries of $A$ where
every column contributes one entry to each product. In order for the
product to be of the type $a_+ e^{ik\vol(\cG_0)}$, each column
corresponding to a variable $\alpha_e$ must contribute the entry
$e^{ik\length(e)}$. Thus, the constant entries of the columns
corresponding to the variables $\alpha_e$ are irrelevant to our
question and can be replaced by zeros; this will not affect the
value of $a^+(\det A)$. Noticing that the columns of $D$
corresponding to the variables $\gamma_1,\dots,\gamma_q$ and $\zeta_v$ are all zeros, we conclude that
$$
a^+(\det A)=a^+(\det A_0), \quad \text{where} \quad A_0 =
\begin{pmatrix}
B&C\\
0&E
\end{pmatrix}.
$$
By a general matrix identity, $\det A_0=\det B\det E$. Finally, a
simple row reduction shows that $\det B=0$; this is easy to see in
the case of \eqref{e4}. Thus, the coefficient $a^+(\det A)$
vanishes. By \eqref{e2}, it follows that $\sigma^+<\vol \cG_0$, as
claimed.

We note (although this is not needed for our proof) that
$\sigma^-=-\vol\cG_0$ both in the balanced and in the unbalanced
case; this will be clear from the next part of the proof.
\subsection{Proof of Theorem~\ref{th1}: the unbalanced case}\label{sec.e4}
Assume that all external vertices are unbalanced. We will prove that
\begin{equation}
\sigma^+=\vol(\cG_0), \quad \sigma^-=-\vol(\cG_0). \label{e5}
\end{equation}
The proof uses the same reduction as \eqref{e3} in
Section~\ref{sec.e3}, but the details are somewhat more complicated,
since now we have to consider \emph{all} external vertices.

We label the external vertices by $v_1$,\dots,$v_m$ where
$m=|V^\exte|$. Let $\cG_r$ denote the graph obtained from $\cG_0$ by
adding all the leads of $\cG$ that have ends in the set $\{
v_1,\ldots,v_r\}$, so that $\cG_m=\cG$. Let $A_r$ denote the
constraint matrix $A$ corresponding to the graph $\cG_r$ and let
$a^\pm_r=a^\pm(\det A_r)$.

By the previous reasoning, the graph $\cG_r$ is Weyl if and only if
$a_r^+\not=0$ and $a_r^-\not=0$. Our claim \eqref{e5} follows
inductively from the following statements:
\begin{enumerate}[1.]
\item The graph $\cG_0$ is Weyl.
\item The coefficient  $a_r^-$ is non-zero for all $r$.
\item For all $r$, if  $a_{r-1}^+\not=0$ then $a_r^+\not=0$.
\end{enumerate}

Item~1 holds because the  operator $H$ on $\cG_0$ has discrete
spectrum and no other resonances. The eigenvalues obey the Weyl law
by a standard variational argument using Dirichlet-Neumann
bracketing.

Let us prove item~3. We reorder the rows and columns of $A_r$ with
reference to $v_r$ as in Section~\ref{sec.e3}. We assume that $p$
internal edges $e_1$,\dots,$e_p$ and $q$ leads $\ell_1$,\dots,
$\ell_q$ are attached to $v_r$, and $q\not=p$. The first $q+p+1$
columns of $A_r$ are those relating to the variables $\gamma_1$,
\dots, $\gamma_q$ (associated with $\ell_1$, \dots, $\ell_q$),
$\alpha_1$, \dots,$\alpha_p$ (associated with $e_1$,\dots,$e_p$),
and $\zeta_r$. The first $q+p+1$ rows of $A_r$ are those relating to
the conditions \eqref{d1} and \eqref{d2} for the vertex $v_r$. As in
Section~\ref{sec.e3}, this allows us to write
\begin{equation}
A_r =
\begin{pmatrix}
B_r&C_r\\
D_r&E_r
\end{pmatrix}
\label{e6}
\end{equation}
where $B_r$ is a $(q+p+1)\times(q+p+1)$ matrix. Writing the matrix
$A_{r-1}$ in the same way with reference to \emph{the same vertex}
$v_r$, we obtain
\begin{equation}
A_{r-1} =
\begin{pmatrix}
\wt B_{r-1}&\wt C_{r-1}\\
\wt D_{r-1}&E_{r}
\end{pmatrix},
\label{e7}
\end{equation}
where $\wt B_{r-1}$ is a $(p+1)\times(p+1)$ matrix. In other words,
$\wt B_{r-1}$, $\wt C_{r-1}$, $\wt D_{r-1}$ are the matrices $B_r$,
$C_r$, $D_r$ with relevant $q$ rows and $q$ columns deleted. The
deleted columns correspond to the variables
$\gamma_1$,\dots,$\gamma_q$, and the deleted rows correspond to the
conditions \eqref{d1} associated with the leads $\ell_1$, \dots,
$\ell_q$. Note that the matrix $E_r$ is the same in \eqref{e6} and
\eqref{e7}.

Next, just as in the argument of Section~\ref{sec.e3}, we notice
that
$$
a_r^+=a^+(\det B_r\det E_r) \quad \text{ and }\quad
a_{r-1}^+=a^+(\det\wt B_{r-1}\det E_r).
$$
Finally, by a simple row reduction we obtain
\begin{align}
\det B_r&=(q-p)z_1\dots z_p, \label{e8}
\\
\det \wt B_{r-1}&=(-p)z_1\dots z_p,
\end{align}
where $z_j=e^{ik\length(e_j)}$. It follows that $a^+_r$ and
$a^+_{r-1}$ differ by a non-zero coefficient $(p-q)/p$. This proves
Item 3.

Let us prove Item 2. Here the argument follows that of the proof of
Item 3, only instead of keeping track of the coefficient of
$e^{ik\vol(\cG_0)}$ we need to keep track of the coefficient of
$e^{-ik\vol(\cG_0)}$, and instead of the variables
$\alpha_1$,\dots,$\alpha_p$ we consider the variables
$\beta_1$,\dots, $\beta_p$. Instead of the coefficient $(q-p)$ in
\eqref{e8} we get $(q+p)$, which never vanishes (even if $v_r$ is
balanced). This proves our claim.

\section{A resolvent identity and its consequences}\label{sec.c}

In order to complete the proof of Theorem~\ref{th1}, it remains to
provide the proof of Theorem~\ref{thm.e1}. Theorem~\ref{thm.c1}
below provides an explicit  formula for the difference
$R^\varkappa(k)-R_D^\varkappa(k)$ in terms of the
Dirichlet-to-Neumann map. This leads immediately to the trace
formula \eqref{c12}, which is the key to our proof of
Theorem~\ref{thm.e1} in Section~\ref{sec.d}. The formulae obtained
in this section are ``complex-scaled'' versions of resolvent
identities well known in the theory of boundary value problems, see
e.g.\ \cite{gesztesy} and references therein.

\subsection{Dirichlet-to-Neumann map}\label{sec.c1}
Throughout this section, we assume that the parameter $k\in\C_+$ is
fixed. Let $\cL(k)$ be as defined in Section~\ref{sec.d1} and let
$\cM(k)=\cL(k)\cap C(\cG)$. Each $f\in\cM(k)$ determines a vector
$\zeta\in \C^{|V|}$ by restriction to $V$. Conversely, every
$\zeta\in\C^{|V|}$ arises from a function $f\in\cM(k)$; this can be
seen by comparing $\dim \cL(k)$ with the number of constraints
imposed by writing $f(v)=\zeta_v$, $v\in V$. Finally, the assumption
$k\in\C_+$ implies that only one function $f\in\cM(k)$ corresponds
to each set of values $\zeta\in\C^{|V|}$ (otherwise we would have a
complex eigenvalue of the operator with Dirichlet boundary
conditions on all vertices). This shows that we may define the
Dirichlet-to-Neumann map $\Lambda(k): \C^{|V|}\to \C^{|V|}$ by
$$
(\Lambda(k)\zeta)_v=N_vf
$$
where $f$ corresponds to $\zeta$ as described above and $N_v$ was
defined in Section~\ref{notation}. This map is a well known tool in
the spectral theory of boundary value problems and has also been
used in quantum graph theory \cite{Ong,Kuchment2}.

\subsection{The functions $\varphi_v$ and formulae for $\Lambda$}\label{sec.c2}
Given $v\in V$, let $\varphi_v$ be the function in $\cM(k)$ that
satisfies
$$
\varphi_v(u)=\delta_{uv}, \quad \forall u,v\in V.
$$
The functions $\varphi_v$ are given by the following explicit
expressions. Let $v\in e$, $e\in E^\inte$ and identify $e$ with
$[0,\length]$ where $v$ corresponds to the point $0$. Then
\begin{equation}
\varphi_v(x)=\frac{\sin k(\length-x)}{\sin k\length}, \quad
x\in[0,\length]=e. \label{c1}
\end{equation}
In the same way, if $e\in E^\exte$ and $v$ is identified with the
point $0$, then
\begin{equation}
\varphi_v(x)=e^{ikx}, \quad x\in[0,\infty)=e. \label{c2}
\end{equation}
If the dependence on $k$ needs to be emphasized, we will write
$\varphi_v(x;k)$ instead of $\varphi_v(x)$.

\begin{lemma}\label{extralemma}
If $k\in\C_+$ then the map $\Lambda(k)$ is invertible. Its matrix
entries are given by
\begin{align}
\Lambda_{uv}&=0, & \text{if $u\not=v$, $u\not\sim v$;} \label{c3}
\\
\Lambda_{uv}&= \sum_{\genfrac{}{}{0pt}{}{e\in E^\inte}{u,v\in
e}}\frac{k}{\sin k\length(e)}, & \text{if $u\not=v$, $u\sim v$;}
\label{c4}
\\
\Lambda_{vv}&=ikq(v)-k\sum_{\genfrac{}{}{0pt}{}{e\in E^\inte}{v\in
e}} \cot (k\length(e)), & \text{for any $v\in V$;} \label{c5}
\end{align}
where $q(v)$ was defined in Section~\ref{notation}.
\end{lemma}
\begin{proof}
If $\Lambda(k)\zeta=0$, then the corresponding function
$f\in\cM(k)\subset L^2(\cG)$ satisfies the Kirchhoff's boundary
condition at every vertex, which implies that $f\in\Dom H$ and
$Hf=k^2f$. Since $\Spec(H)=[0,\infty)$ and $\Im k>0$, this implies
that $f=0$. Therefore $\Lambda(k)$ is invertible.

By the definition of $\varphi_v$, we have
$$
\Lambda_{uv}=N_u\varphi_v.
$$
The formulae for the matrix entries are obtained by combining this
with \eqref{c1} and \eqref{c2}.
\end{proof}

It follows from Lemma~\ref{extralemma} that $\Lambda(k)$ can be
extended to a meromorphic function of $k\in\C$ whose poles are all
on the real axis, and that for any $u,v\in V$ one has
\begin{equation}
\Lambda_{uv}(k)=\Lambda_{vu}(k) \quad\text{ and }\quad
\overline{\Lambda_{uv}(k)}=\Lambda_{uv}(-\overline{k}), \quad
k\in\C. \label{c5a}
\end{equation}

In the calculations below the expressions $\Lambda_{uv}^{-1}$ will
denote the matrix entries of $(\Lambda(k))^{-1}$.

\subsection{The complex-scaled version of $\varphi_v$}\label{sec.c3}
We will need a version of the functions $\varphi_v$ pertaining to
the ``complex-scaled'' operator $H(\varkappa)$. Let $k\in\C_+$ and
$\varkappa\in\C$ be such that $ke^{\varkappa}\in\C_+$. Given $v\in
V$,we define $\vf$ by
$$
\vf(x;k) =
\begin{cases}
\varphi_v(x;k), & \text{if $x\in\cG_0$;}
\\
\varphi_v(0;k)e^{\varkappa/2}\exp(ike^{\varkappa}x), & \text{if
$x\in \ell=[0,\infty)$, $\ell\in E^\exte$.}
\end{cases}
$$
Clearly, $\vf$ is a solution to the equation $H(\varkappa)\vf
=k^2\vf$ on every edge  of $\cG$. It is also straightforward to see
that $\vf\in\wt C(\cG)$  and $\vf$ satisfies the boundary condition
\eqref{b3} on every external vertex. For $f\in\wt C(\cG)$, let us
denote
$$
N^\varkappa_v f =
\begin{cases}
N_v f, & \text{ if $v\in V^\inte$,}
\\
N_v^\inte f+e^{-3\varkappa/2} N_v^\exte f, & \text{ if $v\in
V^\exte$.}
\end{cases}
$$
It is straightforward to see that
\begin{equation}
\Lambda_{uv}=N_u^\varkappa \vf, \quad \forall u,v\in V, \label{c6}
\end{equation}
where the l.h.s.\ depends on $k$ but not on $\varkappa$. Moreover
\begin{equation}
\overline{\vf(x;k)}=\varphi_v^{\overline{\varkappa}}(x;-\overline{k}).
\label{c6a}
\end{equation}

\subsection{The resolvent identity}\label{sec.c4}

Let $H_D$ be the self-adjoint operator in $L^2(\cG)$ defined by
$H_Df=-f''$ with a Dirichlet boundary condition at every vertex of
$\cG$. Given $\varkappa\in\C$, we define the ``complex-scaled''
version of $H_D$ as follows; $H_D(\varkappa)$ is the operator acting
in $L^2(\cG)$ defined by
$$
(H_D(\varkappa)f)(x) =
\begin{cases}
-f''(x), &\text{if $x\in\cG_0$},
\\
-e^{-2\varkappa}f''(x), &\text{if $x\in\cG_\infty$,}
\end{cases}
$$
with a Dirichlet boundary condition at every vertex of $\cG$. Of
course, $H_D(\varkappa)$ splits into an orthogonal sum of operators
acting on $L^2(e)$ for all $e\in E$. We see immediately that in
addition to its essential spectrum $e^{-2\varkappa}[0,\infty)$, the
operator $H_D(\varkappa)$ has a discrete set of positive eigenvalues
with finite multiplicities.

We set
$$
R_D^\varkappa(k)=(H_D(\varkappa)-k^2 I)^{-1}, \quad
R^\varkappa(k)=(H(\varkappa)-k^2 I)^{-1},
$$
whenever the inverse operators exist. We denote by $R^\varkappa(k;
x,y)$ (resp. $R^\varkappa_D(k;x,y)$), $x,y\in\cG$, the integral
kernel of the resolvent $R^\varkappa(k)$ (resp. of
$R_D^\varkappa(k)$).

The fact that $H_D(\varkappa)$ and $H(\varkappa)$ coincide except
for different boundary conditions at each of the $|V|$ vertices
indicates that the difference of the two resolvents should have rank
$|V|$. Our next theorem makes this explicit. Formulae of this type
are well known in the theory of boundary value problems; see e.g.\
\cite{gesztesy} and references therein. In the context of graphs,
similar considerations have been used in \cite{Kostrykin1,Kostrykin2,Kostrykin3,Ong}.

\begin{theorem}\label{thm.c1}
For any $k\in\C_+$ and any $\varkappa\in\C$, such that
$ke^{\varkappa}\in\C_+$, we have
\begin{equation}
R^\varkappa(k;x,y) - R^\varkappa_D(k;x,y) =- \sum_{u,v\in V}
\Lambda_{uv}^{-1}(k)\varphi_v^\varkappa(x;k)\varphi_u^\varkappa(y;k),
\label{c7}
\end{equation}
for any $x,y\in\cG$.
\end{theorem}
\begin{proof}
1. Let $\wt R^\varkappa(k)$ be the operator in $L^2(\cG)$ with the
integral kernel given by
$$
\wt R^\varkappa(k;x,y) = R^\varkappa_D(k;x,y) - \sum_{u,v\in V}
\Lambda_{uv}^{-1}(k)\varphi_v^\varkappa(x;k)\varphi_u^\varkappa(y;k).
$$
We need to prove that $\wt R^\varkappa(k)$ is a bounded operator,
that it maps $L^2(\cG)$ into $\Dom H(\varkappa)$ and that the
identities
\begin{align}
(H(\varkappa)-k^2I)\wt R^\varkappa(k)&=I \label{c8}
\\
\wt R^\varkappa(k)(H(\varkappa)-k^2I)&=I \label{c9}
\end{align}
hold true. First note that since $\vf$ decays exponentially on all
leads, the boundedness of $\wt R^\varkappa(k)$ is obvious. Next,
using \eqref{c5a}, \eqref{c6a} one obtains $\wt R^\varkappa(k)^*=\wt
R^{\overline{\varkappa}}(-\overline{k})$. From here and \eqref{b5}
by taking adjoints we see that \eqref{c9} is equivalent to
$$
(H(\overline{\varkappa})-(-\overline{k})^2)\wt
R^{\overline{\varkappa}}(-\overline{k})=I
$$
which is \eqref{c8} with $-\overline{k}$, $\overline{\varkappa}$
instead of $k$, $\varkappa$. We note that $k\in\C_+$,
$ke^{\varkappa}\in\C_+$ if and only if $-\overline{k}\in\C_+$,
$-\overline{k}e^{\overline{\varkappa}}\in\C_+$. Thus, \eqref{c9}
follows from \eqref{c8}.

2. It suffices to prove that for a dense set of elements $f\in
L^2(\cG)$, the inclusion $\wt R^\varkappa(k)f\in\Dom H(\varkappa)$
and the identity
\begin{equation}
(H(\varkappa)-k^2I)\wt R^\varkappa(k)f=f \label{c10}
\end{equation}
hold true. Let $f$ be from the dense set of all continuous functions
compactly supported on $\cG$ and vanishing near all vertices of
$\cG$. Let us check that the function $g=\wt R^\varkappa(k)f$
belongs to $\Dom H(\varkappa)$. It is clear that the restriction of
$g$ onto any edge  $e$ of $\cG$ belongs to  the Sobolev space
$W^2_2(e)$. Thus, it suffices to check that $g$ belongs to $\wt
C(\cG)$ and satisfies the boundary conditions \eqref{b3} and
\eqref{b4}.

Denote $g_0=R^\varkappa_D(k)f$. Since $g_0\in\Dom H_D(\varkappa)$,
$g_0$ vanishes on all vertices. Therefore $g_0$ lies in $\wt C(\cG)$
and satisfies \eqref{b3} at every external vertex $v$. As mentioned
in Section~\ref{sec.c3}, the functions $\varphi_v^\varkappa$ also
belong to $\wt C(\cG)$ and satisfy \eqref{b3} at every external
vertex $v$. Thus, $g$ also has these properties.

Our next task is to prove that the boundary condition \eqref{b4} is
satisfied for the function $g$. Suppose that $f$ is supported on a
single edge, which we identify with $[0,\length]$. Then the integral
kernel of $R_D^\varkappa(k)$ can be explicitly calculated, which
gives
$$
g_0'(0)=\int_0^\length \frac{\sin k(\length-x)}{\sin k\length}
f(x)dx.
$$
Similarly, if $f$ is supported on a lead $[0,\infty)$, then a direct
calculation shows that
$$
g_0'(0)=e^{2\varkappa} \int_0^\infty \exp(ike^{\varkappa}x)f(x)dx.
$$
Combining this, we see that for any $w\in V^\exte$ we have
$$
N_w^\varkappa g_0 = \int_\cG f(x)\varphi_w^\varkappa(x)dx.
$$
Using the last identity and \eqref{c6}, for any $w\in V^\exte$ we
get:
$$
N_w^\varkappa g = \int_\cG f(x)\varphi_w^\varkappa(x)dx -
\sum_{u,v\in V} \Lambda^{-1}_{uv}\Lambda_{wv} \int_\cG f(x)
\varphi_u^\varkappa(x)dx =0,
$$
and so the boundary condition \eqref{b4} is satisfied for $g$. Thus,
$g\in \Dom H(\varkappa)$, as required.

3. It remains to note that the identity \eqref{c10} follows from the
fact that $R_D^\varkappa$ is the resolvent of $H_D(\varkappa)$ and
the fact that $\vf$ satisfies the equation $H(\varkappa)\vf =k^2\vf$
on every edge and lead of $\cG$.
\end{proof}

\subsection{A trace formula}\label{sec.c7}
The trace formula \eqref{c12} below results by calculating the
traces of both sides of \eqref{c7}. Since the r.h.s.\ of \eqref{c7}
is a finite rank operator, the trace is well defined; the fact that
the value of \eqref{c12} does not depend on $\varkappa$ can be
proved by complex scaling, but the direct proof is almost as easy.

The  identity \eqref{c12} below can be rephrased by saying that the
(modified) perturbation determinant of the pair of operators
$H(\varkappa)$, $H_D(\varkappa)$ equals $\det\Lambda(k)$. Statements
of this nature (for $\varkappa=0$) are well known in the theory of
boundary value problems; see e.g.\ \cite{Carron} and references
therein. The key to our proof of Theorem~\ref{thm.e1} will be
\eqref{c12} and Lemma~\ref{lma.d1}, in which $\det A(k)$ and
$\det\Lambda(k)$ are related.

\begin{theorem}\label{thm.c2}
For any $k\in\C_+$ and any $\varkappa\in\C$, such that
$ke^{\varkappa}\in\C_+$, we have
\begin{equation}
\Tr(R^\varkappa(k)-R_D^\varkappa(k)) = -\frac{\frac{d}{dk}\det
\Lambda(k)}{2k\det\Lambda(k)}. \label{c12}
\end{equation}
In particular, the l.h.s.\ is independent of $\varkappa$.
\end{theorem}
\begin{proof}
1. Theorem~\ref{thm.c1} yields
\begin{equation}
\Tr(R^\varkappa(k)-R^\varkappa_D(k)) = -\sum_{u,v\in V}
\Lambda_{uv}^{-1}(k) \sigma_{uv}^\varkappa(k), \label{c13}
\end{equation}
where
\begin{equation}
\sigma_{uv}^\varkappa(k) = \int_\cG
\varphi^\varkappa_u(x;k)\varphi^\varkappa_v(x;k) dx. \label{c11}
\end{equation}
We next compute the coefficients $\sigma_{uv}$ explicitly. If
$v\not=u$ and $v\not\sim u$ then $\supp \varphi_v^\varkappa\cap
\supp \varphi_u^\varkappa=\varnothing$ and so $\sigma_{uv}=0$. If
$v\not=u$ and $v\sim u$ then by \eqref{c1}
\begin{multline*}
\hspace*{3em}\sigma_{uv} = \sum_{\genfrac{}{}{0pt}{}{e\in
E^\inte}{u,v\in e}} \int_0^\length \frac{\sin kx}{\sin k\length(e)}
\frac{\sin k(\length(e)-x)}{\sin k\length(e)}dx
\\
= \frac1{2k} \sum_{\genfrac{}{}{0pt}{}{e\in E^\inte}{u,v\in e}}
\frac{\sin k\length(e)-k\length(e) \cos k\length(e)}{(\sin
k\length(e))^2},\hspace*{3em}
\end{multline*}
and finally,
\begin{multline*}
\hspace*{3em}\sigma_{vv} = \sum_{\genfrac{}{}{0pt}{}{e\in
E^\inte}{v\in e}} \int_0^{\length(e)} \left(\frac{\sin kx}{\sin
k\length(e)}\right)^2 dx + q(v) \int_0^\infty
(e^{\varkappa/2}\exp(ike^{\varkappa}x))^2 dx
\\
= \frac1{2k} \sum_{\genfrac{}{}{0pt}{}{e\in E^\inte}{v\in e}}
\frac{k\length(e)-\cos k\length(e)\sin k\length(e)}{(\sin
k\length(e))^2} + \frac{i}{2k}q(v).\hspace*{3em}
\end{multline*}

2. Noting that $\sigma_{uv}$ depend on $k$ but not on $\varkappa$, a
direct calculation using \eqref{c3}--\eqref{c5} yields
$$
\frac{1}{2k}\frac{d}{dk} \Lambda_{uv}(k)=\sigma_{uv}(k).
$$
It follows that
\begin{multline*}
\hspace*{2em}\Tr(R^\varkappa(k)-R_D^\varkappa(k)) = -\sum_{u,v\in V}
\Lambda_{uv}^{-1}(k)\frac{1}{2k}\frac{d}{dk}\Lambda_{uv}(k)
\\
= - \frac{1}{2k}\Tr(\Lambda^{-1}(k)\frac{d}{dk}\Lambda(k)) =
-\frac{\frac{d}{dk}\det \Lambda(k)}{2k\det\Lambda(k)},\hspace*{2em}
\end{multline*}
as required.
\end{proof}

\section{Proof of Theorems~\ref{thm.e1} and \ref{thm.b2}}\label{sec.d}

\subsection{ Calculation of $\det A(k)$ }\label{sec.d5}
Given $k\in\C$, we define
\begin{equation}
\delta(k)=\prod_{e\in E^\inte} (k\sin k\length(e)). \label{d4}
\end{equation}
Let $A(k)$ be the matrix defined in Section~\ref{sec.d1}.
\begin{lemma}\label{lma.d1}
For any $k\in\C_+$, we have the identity
\begin{equation}
\det A(k) = \pm
\frac{2^{\abs{E^\inte}}i^{\abs{E^\inte}-\abs{V}}}{k^{\abs{E^\inte}+\abs{V}}}
\delta(k)\det\Lambda(k), \label{d6}
\end{equation}
where the sign $\pm$ depends on the ordering of the rows and columns
of the matrix $A(k)$.
\end{lemma}

\begin{proof}
1. Let us order the rows and the columns of $A(k)$ in such a way
that the first $\abs{V}$ rows correspond to the conditions
$N_v(u)=0$, and the first $\abs{V}$ columns correspond to the
variables $\zeta$. Then $A(k)$ can be written in the block form as
\begin{equation}
A=
\begin{pmatrix}
0 & M
\\
-N & P
\end{pmatrix}
\label{d7}
\end{equation}
where $0$ is the $\abs{V}\times\abs{V}$ zero matrix and $P$ is a
$(2\abs{E^\inte}+\abs{E^\exte})\times(2\abs{E^\inte}+\abs{E^\exte})$
matrix. The elements of $N$ are $0$ or $1$, the elements of $M$ are
$0$, $\pm 1$, $\pm e^{\pm ik\length}$, and the elements of $P$ are
$0$, $\pm 1$, or $e^{\pm ik\length}$. For example, the matrix
\eqref{d2a} is written in this form.

2. Let us reorder the rows of $P$ in such a way that any two
constraints associated with the continuity conditions at the two
endpoints of the same edge follow one another. Let us also reorder
the columns of $P$ such that each variable $\beta_e$ follows the
corresponding variable $\alpha_e$. For example, the block $P$ of the
matrix \eqref{d2a} after such reordering will be
$$
\left(\begin{array}{ccccc}
1 & 1 & 0 & 0 & 0 \\
{z_1} & z_1^{-1} & 0 & 0 & 0 \\
0 & 0 & 1 & 1 & 0 \\
0 & 0 & z_2 & z_2^{-1} & 0 \\
0 & 0 & 0 & 0 & 1
\end{array}\right).
$$
In general, after this reordering, $P$ assumes a block-diagonal
structure with blocks either of size $2\times 2$ with elements
$$
\begin{pmatrix}
1 & 1
\\
e^{ik\length} & e^{-ik\length}
\end{pmatrix}
$$
or of size $1\times 1$ with the element $1$. From here it follows
that
\begin{equation}
\det P = \pm\prod_{e\in E^\inte}(2i\sin(k\length(e))) =
\pm(2i)^{\abs{E^\inte}}k^{-\abs{E^\inte}}\delta(k). \label{d8}
\end{equation}
In particular, since $k\in\C_+$, the matrix $P$ is invertible.

3. By applying the Schur complement method to \eqref{d7} one obtains
\begin{equation}
\det A=\det P \det(MP^{-1}N). \label{d9}
\end{equation}
Let us prove that
\begin{equation}
ik MP^{-1}N=\Lambda(k). \label{d10}
\end{equation}
Let $\zeta\in\C^{\abs{V}}$ and let $a=P^{-1}N\zeta$. The vector $a$
represents a set of parameters $\alpha$, $\beta$, $\gamma$. Let
$f\in\cL(k)$ be the solution with this set of parameters. The
equation $Pa=N\zeta$ implies that the solution $f$ is continuous on
$\cG$ and satisfies $f(v)=\zeta_v$ for any vertex $v$. Next, the
coordinates of the vector $ik MP^{-1}N\zeta=ik Ma$ are given by
$$
ik(Ma)_v=N_v f.
$$
This shows that $ikMa=\Lambda(k)\zeta$, as required.

4. By combining \eqref{d8}--\eqref{d10} one obtains
\begin{multline*}
\hspace*{3em}\det A(k) = \det P(k) \det(M(k)P^{-1}(k)N(k))
\\
=
\pm(2i)^{\abs{E^\inte}}k^{-\abs{E^\inte}}\delta(k)\det((ik)^{-1}\Lambda(k)),\hspace*{3em}
\end{multline*}
which yields \eqref{d6} immediately.
\end{proof}

\subsection{Proof of Theorem~\ref{thm.e1}}
1. Let $k\in\C_+$ and let $\chi_0$ and $\chi_\infty$ be defined as
in Section~\ref{notation}. Clearly, $\chi_0 R_D(k)\chi_0$ is an
orthogonal sum of resolvents of the operators $-d^2/dx^2$ on the
intervals $(0,\length(e))$, $e\in E^\inte$, with the Dirichlet
boundary conditions. For each such operator we have that
$(-d^2/dx^2-k^2)^{-1}$ is trace class and
\begin{multline*}
\Tr(-d^2/dx^2-k^2)^{-1} = \sum_{n=0}^\infty\left((\pi
n/\length)^2-k^2\right)^{-1}
\\
= -\frac{1}{2k^2} -\frac1{2k}\sum_{n=-\infty}^\infty \frac1{k-\pi
n/\length} = -\frac{1}{2k^2} -\frac{\length}{2k}\cot (k\length) =
-\frac{\frac{d}{dk}(k\sin(k\length))}{2k(k\sin(k\length))}.
\end{multline*}
Summing over all edges, a direct calculation shows that $\chi_0
R_D(k)\chi_0$ is a trace class operator and
\begin{equation}
\Tr(\chi_0 R_D(k)\chi_0) = -\frac{\frac{d}{dk}\delta(k)}{2k
\delta(k)}. \label{d5}
\end{equation}

2. Let $k\in\C_+$, $ke^\varkappa\in\C_+$. It is easy to see that the
resolvent $R^\varkappa_D(k)$ commutes with $\chi_0$, $\chi_\infty$
and that
$$\chi_0
R_D^\varkappa(k)\chi_0=\chi_0 R_D(k)\chi_0.
$$
Therefore we have
\begin{equation}
R^\varkappa(k)-\chi_\infty R_D^\varkappa(k)\chi_\infty =
R^\varkappa(k)- R_D^\varkappa(k) + \chi_0 R_D(k)\chi_0.
\label{resident}
\end{equation}
By combining Theorem~\ref{thm.c2} and \eqref{resident}, we obtain
\begin{equation}
\Tr(R^\varkappa(k)-\chi_\infty R_D^\varkappa(k)\chi_\infty) =
-\frac{\frac{d}{dk}\det \Lambda(k)}{2k\det\Lambda(k)}
-\frac{\frac{d}{dk}\delta(k)}{2k \delta(k)} =
-\frac{\frac{d}{dk}(\delta(k)\det \Lambda(k))}{2k
\delta(k)\det\Lambda(k)}. \label{d11}
\end{equation}
Using Lemma~\ref{lma.d1}, we then obtain
\begin{equation}
\Tr(R^\varkappa(k)-\chi_\infty R_D^\varkappa(k)\chi_\infty) =
\frac{\abs{E^\inte}+\abs{V}}{2k^2} - \frac{\frac{d}{dk}\det A(k)}{2k
\det A(k)}, \label{d3}
\end{equation}
for all $k\in\C_+$ and $ke^\varkappa\in\C_+$.

3. The r.h.s.\ of \eqref{d3} is a single-valued meromorphic function
of $k\in\C$. Let $\tau^\varkappa(k)$ be the l.h.s.\ of \eqref{d3}.
For each fixed $\varkappa\in\C$, the function $\tau^\varkappa(k)$ is
meromorphic in $\C$ with the cut along the line determined by the
condition
$k^2\in\sigma_\ess(H(\varkappa))=e^{-2\varkappa}[0,\infty)$. In
other words, $\tau^\varkappa$ is meromorphic and single-valued in
each of the two half-planes $\Im ke^{\varkappa}>0$ and $\Im
ke^{\varkappa}<0$. By the uniqueness of analytic continuation, for
each $\varkappa$ the identity \eqref{d3} extends to all $k$ such
that $\Im ke^\varkappa>0$.

4. Let $k_0\in\cR$ with the algebraic  multiplicity $m(k_0)\geq1$
and let $\theta\in(0,\pi)$ with $-\theta<\arg k_0\leq0$. Then $\Im
k_0e^{i\theta}>0$ and so the identity \eqref{d3} with
$\varkappa=i\theta$ holds for all $k$ near $k_0$. If $\gamma$ is a
sufficiently small circle with centre at $k_0$, then the
multiplicity $m(k_0)$ equals the rank, or equivalently the trace, of
the Riesz spectral projection
\begin{equation}
P^\theta(k_0) = -\frac{1}{2\pi i} \int_\gamma R^{i\theta}(k)2kdk.
\label{b6}
\end{equation}
Next, since the operator $H_D(i\theta)$ restricted to
$L^2(\cG_\infty)$ has no eigenvalues, the operator valued function
$\chi_\infty R_D^{i\theta}(k)\chi_\infty$ is analytic for $\Im
ke^{i\theta}\not=0$. It follows that
$$
-\frac{1}{2\pi i} \int_\gamma \chi_\infty
R_D^{i\theta}(k)\chi_\infty 2kdk=0.
$$
By taking the trace of the difference of the last two equations and
using \eqref{d3} we obtain
\begin{eqnarray*}
m(k_0)&=& -\frac{1}{2\pi i} \int_\gamma \Tr(R^{i\theta}(k)-\chi_\infty R_D^{i\theta}(k)\chi_\infty )2kdk\\
&=& \frac{1}{2\pi i} \int_\gamma \frac{\frac{d}{dk}\det A(k)}{\det
A(k)} dk.
\end{eqnarray*}
Therefore $m(k_0)$ equals the order of the zero of $\det A(k)$ at
$k=k_0$, as required. \qed

\subsection{Proof of Theorem~\ref{thm.b2}}\label{sec.c5}
This theorem is well known and the proof is presented here for the
sake of completeness.

1. First note that by Theorem~\ref{thm.c1}, the difference of the
resolvents of $H(\varkappa)$ and $H_D(\varkappa)$ is a finite rank
operator. By Weyl's theorem on the invariance of the essential
spectrum under a relatively  compact perturbation we obtain
$$
\sigma_\ess(H(\varkappa))=\sigma_\ess(H_D(\varkappa))=e^{-2\varkappa}[0,\infty).
$$

2. The fact  that the family $H(\varkappa)$ is analytic in the sense
of Kato follows again from Theorem~\ref{thm.c1}, since
$H_D(\varkappa)$ is analytic in the sense of Kato and each of the
functions $\vf$ is analytic in $\varkappa$.

3. The identity \eqref{a3} can be checked by a direct calculation.

4. Let $k\in\cR$ and let $f$ be the corresponding eigenfunction. For
any $\theta\in(0,\pi)$ with $-\theta<\arg k\leq0$, let $f_\theta$ be
the function defined \emph{formally} by $f_\theta=U(i\theta) f$.
More precisely, we set $f_\theta=f$ on $\cG_0$ and
\begin{equation}
f_\theta(x) = f(0)e^{i\theta/2} \exp(ike^{i\theta}x) \label{d12}
\end{equation}
for $x$ on any lead $\ell=[0,\infty)$. By the choice of $\theta$, we
have $\Im ke^{i\theta}>0$ and so $f_\theta\in L^2(\cG)$. A
straightforward inspection shows that $f_\theta\in\Dom H(i\theta)$
and $H(i\theta)f_\theta=k^2 f_\theta$.

5. Conversely, let $\lambda\not\in e^{-2i\theta}[0,\infty)$ be an
eigenvalue of $H(i\theta)$ for $\theta\in(0,\pi)$. Write
$\lambda=k^2$ with $\Im ke^{i\theta}>0$. Then, for the corresponding
eigenfunction $g$ of $H(i\theta)$ we have
$g(x)=g(0)\exp(ike^{i\theta}x)$ on any lead of $\cG$. A direct
inspection shows that $g=f_\theta$ in the same sense as \eqref{d12},
where $f$ is a resonance eigenfunction. Thus, $k\in\cR$ and in
particular, $\Im k\leq 0$. It follows that $2\pi-2\theta<\arg
k^2\leq 2\pi$. \qed

\section{An example}\label{sec.f}

Here we consider resonances of  a particular simple graph $\cG(c)$,
where $c\in[0,1]$ is a certain geometric parameter. The graph
$\cG(c)$ was also considered in \cite[Section~4]{EL2}, but with
different boundary conditions at the vertices. The graph $\cG(c)$ is
Weyl for $c<1$ and non-Weyl for $c=1$. This section has two goals.
The first one is to discuss  the transition between the Weyl and the
non-Weyl cases in order to throw new light on the failure of the
Weyl law. Our second goal is to obtain rigorous bounds on the
locations of individual resonances of $\cG(c)$, which was not
addressed in \cite{EL2}.

\subsection{Definition of $\cG(c)$}
Given $c\in[0,1)$, we consider the graph $\cG_0(c)$ which consists
of two vertices $v_1$ and $v_2$ and two edges $e_1=[0,\length_1]$,
$\length_1=(1-c)\pi$, and $e_2=[0,\length_2]$, $\length_2=(1+c)\pi$.
The vertex $v_2$ is identified with the point $0$ of $e_1$ and with
the point $0$ of $e_2$, and the vertex $v_1$ is identified with the
point $\length_1$ of $e_1$ and with the point $\length_2$ of $e_2$.
Thus, the graph $\cG_0(c)$ is simply a circle with the circumference
$\vol \cG_0(c)=2\pi$ for all $c$. We attach a lead $\ell_1$ at $v_1$
and a lead $\ell_2$ at $v_2$ and denote the thus extended graph by
$\cG(c)$. Geometrically, $\cG(c)$ is a circle with two leads
attached to it. Finally, for $c=1$, let $\cG(c)$ be the circle  of
length $2\pi$ with two leads attached at the same point.

We will denote by $H(c)$  the operator $-\frac{d^2}{dx^2}$ acting in
$L^2(\cG(c))$ subject to the usual continuity and Kirchhoff's
boundary conditions  at the vertices $v_1$ and $v_2$. By
Theorem~\ref{th1}, the graph $\cG(c)$ is Weyl if and only if $c<1$.
At the same time, the graph $\cG(1)$ can be regarded as the limit of
$\cG(c)$ as $c\to 1$ in an obvious geometric sense, so we need to
explain what happens to resonances as $c \to 1$. As we will see,
roughly speaking, half of the resonances of $H(c)$ move off to
infinity as $c\to1$. We will obtain bounds on the curves along which
the resonances move as $c$ increases from $0$ to $1$.

\subsection{The matrix $A(k,c)$ for $\cG(c)$}
Let us display the constraints \eqref{d2b} corresponding to the
graph $\cG(c)$; the matrix $A(k,c)$ will be built up of the rows
corresponding to these constraints. We denote
$z_j=e^{ik\length_j/2}$, $j=1,2$. The constraints corresponding to
the vertex $v_1$ are
\begin{align}
\alpha_1 z_1^2+\beta_1 z_1^{-2}-\zeta_1&=0 \tag{$R_1$}
\\
\alpha_2 z_2^2+\beta_2 z_2^{-2}-\zeta_1&=0 \tag{$R_2$}
\\
\gamma_1-\zeta_1&=0 \tag{$R_3$}
\\
-\alpha_1 z_1^2+\beta_1 z_1^{-2}-\alpha_2 z_2^2 + \beta_2
z_2^{-2}+\gamma_1&=0. \tag{$R_4$}
\end{align}
The first three constraints above are the continuity conditions, and
the last one is the requirement that the sum of the outgoing
derivatives vanishes. Similarly, the constraints corresponding to
the vertex $v_2$ are
\begin{align}
\alpha_1+\beta_1 -\zeta_2&=0 \tag{$R_5$}
\\
\alpha_2 +\beta_2 -\zeta_2&=0 \tag{$R_6$}
\\
\gamma_2-\zeta_2&=0 \tag{$R_7$}
\\
\alpha_1-\beta_1+\alpha_2 -\beta_2+\gamma_2&=0. \tag{$R_8$}
\end{align}
We list these constraints in the order $R_1$, $R_5$, $R_2$, $R_6$,
$R_3$, $R_7$, $R_4$, $R_8$, and order the variables as $\alpha_1$,
$\beta_1$, $\alpha_2$, $\beta_2$, $\gamma_1$, $\gamma_2$, $\zeta_1$,
$\zeta_2$. This leads to the matrix
$$
A(k,c)= \left(\begin{array}{cccccccc}
z_1^2 & z_1^{-2} & 0 & 0 & 0 & 0 & -1 & 0 \\
1 & 1 & 0 & 0 & 0 & 0 & 0 & -1 \\
0 & 0 & z_2^2 & z_2^{-2} & 0 & 0 & -1 & 0 \\
0 & 0 & 1 & 1 & 0 & 0 & 0 & -1 \\
0 & 0 & 0 & 0 & 1 & 0 & -1 & 0 \\
0 & 0 & 0 & 0 & 0 & 1 & 0 & -1 \\
-z_1^2 & z_1^{-2} & -z_2^2 & z_2^{-2} & 1 & 0 & 0 & 0 \\
1 & -1 & 1 & -1 & 0 & 1 & 0 & 0
\end{array}\right).
$$

\subsection{Calculation of $\det A(k,c)$}
The graph $\cG(c)$ has a reflection symmetry with respect to the
midpoints of $e_1$ and $e_2$. This allows to decompose the space
$\cL(k)$ into the direct sum of the subspaces corresponding to even
and odd functions with respect to this symmetry. We use this
decomposition to represent the matrix $A(k,c)$ in a block-diagonal
form where the blocks correspond to the even and odd solutions. More
precisely, let $T_1$ and $T_2$ be the matrices
\begin{align*}
T_1&= \left(\begin{array}{cccccccc}
1 & 1 & 0 & 0 & 0 & 0 & 0 & 0 \\
0 & 0 & 1 & 1 & 0 & 0 & 0 & 0 \\
0 & 0 & 0 & 0 & 1 & 1 & 0 & 0 \\
0 & 0 & 0 & 0 & 0 & 0 & 1 & 1 \\
1 & -1 & 0 & 0 & 0 & 0 & 0 & 0 \\
0 & 0 & 1 & -1 & 0 & 0 & 0 & 0 \\
0 & 0 & 0 & 0 & 1 & -1 & 0 & 0 \\
0 & 0 & 0 & 0 & 0 & 0 & 1 & -1
\end{array}\right)
\\
T_2&= \left(\begin{array}{cccccccc}
z_1^{-1} & 0 & 0 & 0 & z_1^{-1} & 0 & 0 & 0 \\
z_1 & 0 & 0 & 0 & -z_1 & 0 & 0 & 0 \\
0 & z_2^{-1} & 0 & 0 & 0 & z_2^{-1} & 0 & 0 \\
0 & z_2 & 0 & 0 & 0 & -z_2 & 0 & 0 \\
0 & 0 & 1 & 0 & 0 & 0 & 1 & 0 \\
0 & 0 & 1 & 0 & 0 & 0 & -1 & 0 \\
0 & 0 & 0 & 1 & 0 & 0 & 0 & 1 \\
0 & 0 & 0 & 1 & 0 & 0 & 0 & -1
\end{array}\right).
\end{align*}
A straightforward calculation shows that $\det T_1=\det T_2=16$.
Next, let $\wt A(k,c)=T_1 A(k,c) T_2$; the reader is invited to
check that the matrix $\wt A(k)$ can be written as
$$
\wt A = 2\begin{pmatrix} \wt A_\even & 0
\\
0 & \wt A_\odd
\end{pmatrix}
$$
with blocks
$$
\wt A_\even = \left(\begin{array}{cccc}
2C_1 & 0 & 0 & -1 \\
0 & 2C_2 & 0 & -1 \\
0 & 0 & 1 & -1 \\
-2iS_1 & -2iS_2 & 1 & 0
\end{array}\right),
\quad \wt A_\odd = \left(\begin{array}{cccc}
2iS_1 & 0 & 0 & -1 \\
0 & 2iS_2 & 0 & -1 \\
0 & 0 & 1 & -1 \\
-2C_1 & -2C_2 & 1 & 0
\end{array}\right),
$$
where we have used the notation $C_j=\cos(k\length_j/2)$,
$S_j=\sin(k\length_j/2)$, $j=1,2$. Straightforward calculations of
$\det(\wt A_\even)$ and $\det(\wt A_\odd)$ now yield
\begin{theorem}\label{detAcalc}
For all $k\in\C$ and all $c\in[0,1)$ one has
\[
\det A(k,c)=4F_\even(k,c)F_\odd(k,c)
\]
where
\begin{eqnarray*}
F_\even(k,c)&=& i\cos(kc\pi)+i\cos(k\pi)+2\sin(k\pi),\\
F_\odd(k,c)&=& i\cos(kc\pi)-i\cos(k\pi)-2\sin(k\pi).
\end{eqnarray*}
\end{theorem}
We will call the zeros of $F_\even(\cdot,c)$ (resp. of
$F_\odd(\cdot,c)$) the even (resp. odd) resonances. It is not
difficult to check that the resonance eigenfunctions which
correspond to even/odd resonances are even/odd with respect to the
symmetry of the graph $\cG(c)$. By Theorem~\ref{thm.b3}, the real
even/odd resonances are actually eigenvalues of $H(c)$ and therefore
we will call them even/odd eigenvalues.

Finally, it is not difficult to check that the resonances of $H(1)$
are given, as expected, by the zeros of $\det A(k,1)$. In fact, in
this case we have $F_\odd(k,1)=-2\sin(k\pi)$ and
\begin{equation}
F_\even(k,1)=2ie^{-ik\pi}\not=0 \quad \forall k\in\C. \label{d13}
\end{equation}
Thus, the resonances of $H(1)$ coincide with the solutions to
$\sin(k\pi)=0$, i.e.\ they are given by $k\in\Z$. By
Theorem~\ref{thm.b3}, these resonances (for $k\not=0$) coincide with
the eigenvalues of $H(1)$ and all of them have the multiplicity one.
This shows that for $c=1$ we have the asymptotics \eqref{a4} with
$W=\pi=\tfrac12\vol\cG_0$.

\subsection{Locating the odd resonances}

\begin{theorem}\label{thm.f1}
\begin{enumerate}[\rm (i)]
\item
For any $c\in[0,1]$ , any $n\in\Z$ and any $y\geq0$ one has
\linebreak $F_\odd(n+\frac12-iy,c)\not=0$.
\item
For any $c\in[0,1]$ and any $k=x-iy$ with $y>\abs{x}/\sqrt{3}$ one
has $F_\odd(k,c)\not=0$.
\end{enumerate}
\end{theorem}
\begin{proof}
(i) By an explicit calculation,
$$
F_\odd(n+\tfrac12-iy,c ) = i\cos((n+\tfrac12-iy)\pi c) +
(-1)^n\sinh(y\pi) -2(-1)^n\cosh(y\pi)=A+B,
$$
where
\begin{align*}
|A|&= |\cos((n+1/2-iy)\pi c )|\\
&=|\cos((n+1/2)\pi c )\cosh(y\pi c )
+i\sin((n+1/2)\pi c )\sinh(y\pi c )|\\
&\leq \cosh(y\pi c ) \leq  \cosh(y\pi)
\end{align*}
and
\[
|B|=2\cosh(y\pi)-\sinh(y\pi)=\cosh(y\pi)+\rme^{-y\pi}.
\]
We deduce that
\[
|F_\odd(n+\tfrac12-iy,c )|\,\geq \,|B|-|A| \,\geq \, \rme^{-y\pi}\,
>\, 0.
\]
(ii) We start by observing that $|F_\odd(k,c)|\geq 2A-B$ where
\begin{eqnarray*}
A&=&|\sin(k\pi)|\\
B&=&|\cos(k\pi)-\cos(k\pi c)| =\left|\int_c^1 k\pi \sin(k\pi
s)\,\rmd s \right|.
\end{eqnarray*}
If $u\in \R$ and $v\geq 0$ then
\[
\sin(u-iv)=\sin(u)\cosh(v)-i\cos(u)\sinh(v).
\]
Therefore
\[
\sinh(v)\leq |\sin(u-iv)|\leq \cosh(v).
\]
We deduce that $A\geq \sinh(y\pi)$ and
$$
B\leq \int_c^1 |k|\pi \cosh(y\pi s)\,\rmd s =\frac{|k|}{y} \left(
\sinh(y\pi)-\sinh(y\pi s)\right) \leq \frac{|k|}{y} \sinh(y\pi).
$$
These bounds imply that $2A-B>0$ if $2y>|k|$, which yields the
theorem immediately.
\end{proof}
It follows that all odd resonances are located in the rectangles
$$
\Pi^\odd_n=\left\{x-iy: \abs{x-n}<\tfrac12, \ 0\leq y\leq
\tfrac{2\abs{n}+1}{2\sqrt{3}} \right\}, \quad n\in \Z.
$$
The following statement, in combination with Rouchet's theorem,
shows that each of the rectangles $\Pi^\odd_n$ contains exactly one
odd resonance of algebraic multiplicity one for all $c\in[0,1]$.

\begin{theorem}\label{cequals0}
If $c =0$ there is a resonance of algebraic multiplicity one at
$k=n-i\log(3)/\pi$ for every odd $n\in\Z$ and an eigenvalue of
multiplicity one at $k=n$ for every non-zero even $n\in\Z$. There is
also a resonance of algebraic multiplicity one at $k=0$. No other
odd resonances or eigenvalues exist if $c =0$.
\end{theorem}
The proof follows from the explicit formula
$$
F_\odd(k,0)=\frac{i}2(e^{ik\pi} +3)(1-e^{-ik\pi}).
$$
By the implicit function theorem, we obtain that each of the zeros
of $F_\odd(\cdot;c)$ is a real analytic function of $c\in[0,1]$ with
values in $\Pi^\odd_n$. The set of all odd resonances is therefore
the union of a sequence of bounded real analytic curves.

It is interesting to note that each of these resonance curves
intersects the real axis, thereby (by Theorem~\ref{thm.b3}) giving
rise to embedded eigenvalues. This happens at rational values of
$c$. More precisely, a direct computation shows that $F_\odd(k,c)=0$
for $k\in\R$ if and only if
$$
k=m+n, \quad c=\frac{m-n}{m+n} \quad \text{for some $m,n\in\N$.}
$$

Figure~\ref{figure1} plots a typical odd resonance curve as $c$ increases from $0$ to $1$. The curve starts at $7-i\log(3)/\pi$ when $c=0$ and passes through $7$ when $c=\frac{1}{7},\, \frac{3}{7},\, \frac{5}{7},\, 1$.

\begin{figure}[h!]
\begin{center}\hspace*{-2em}\includegraphics[width=6in]{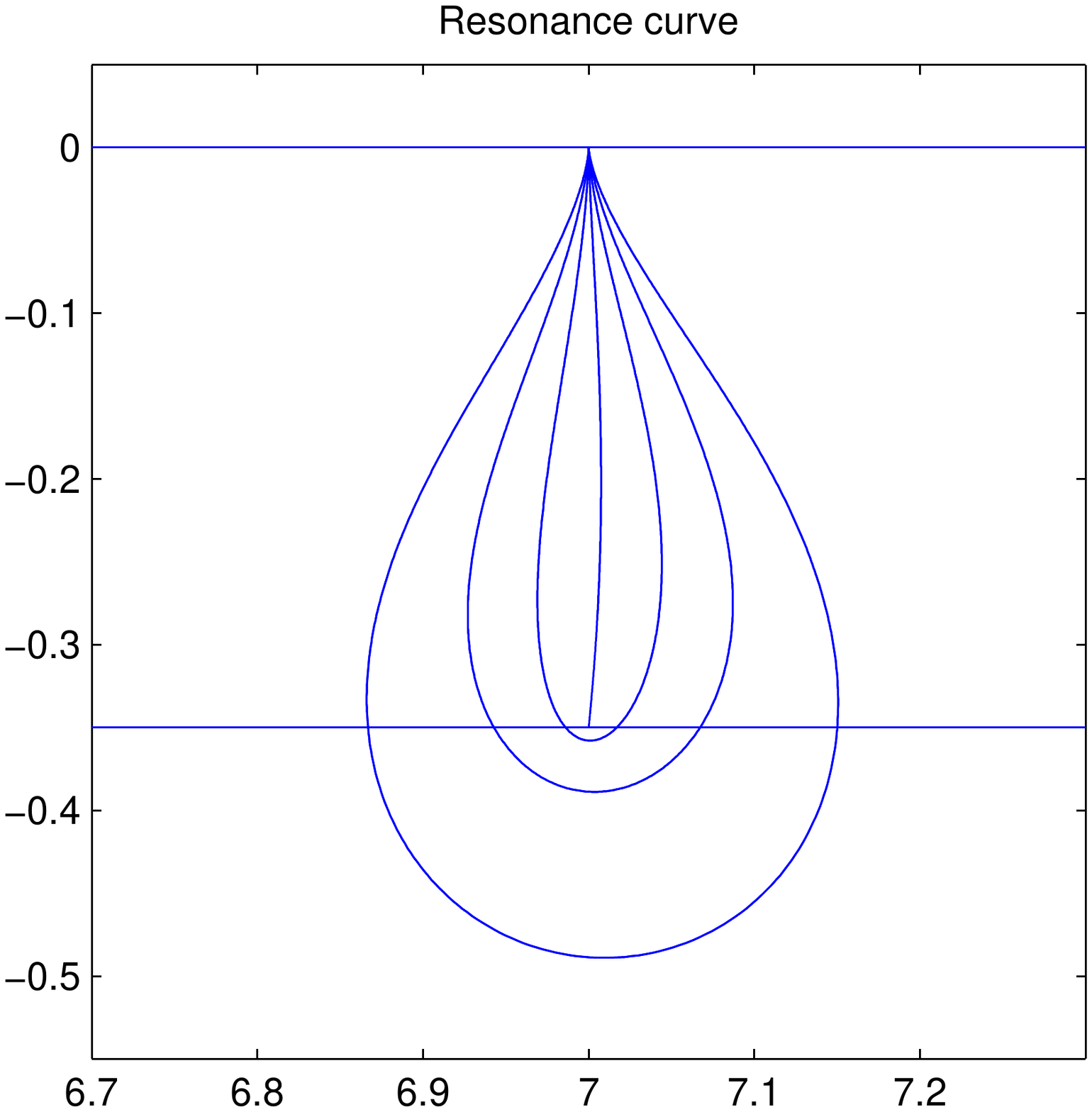}
\end{center}
\vspace*{-4ex} \caption{The odd resonance curve in $\Pi^\odd_7$}
\label{figure1}
\end{figure}

\subsection{Locating the even resonances}
\begin{theorem}\label{thm.f2}
\begin{enumerate}[\rm (i)]
\item
For any $c\in[0,1]$, any $n\in\Z$ and  any $y\geq0$ one has
\linebreak $F_\even(n+\frac12-iy,c)\not=0$.
\item
For any $c\in[0,1)$ and any $k=x-iy$ with $y>\frac{\log
3}{\pi(1-\abs{c})}$, one has $F_\odd(k,c)\not=0$.
\end{enumerate}
\end{theorem}
\begin{proof}
(i) We have
$$
F_\even(n+1/2-iy,c )=A-B,
$$
where $A$, $B$ are as in the proof of Theorem~\ref{thm.f1}(i). The
rest of the proof is the same as in Theorem~\ref{thm.f1}(i).

(ii) For any $k=x-iy$ we have
\begin{align}
\frac{1}{2}\rme^{y\pi|c |}+\frac{1}{2} &\geq \cosh(y\pi c ) \geq
\left| \cos(x\pi c )\cosh(y\pi c )+i\sin(x\pi c )\cosh(y\pi c
)\right| \notag
\\
&\geq \left| \cos(x\pi c )\cosh(y\pi c )+i\sin(x\pi c )\sinh(y\pi c
)\right| =|\cos(k\pi c )| \label{f1}
\end{align}
and
\begin{equation}
|i\cos(k\pi)+2\sin(k\pi)| \geq\frac{1}{2}\left|\rme^{ik\pi}\right|
-\frac{3}{2}\left|\rme^{-ik\pi}\right| =\frac{1}{2}\rme^{y\pi}
-\frac{3}{2}\rme^{-y\pi}. \label{f2}
\end{equation}
Now suppose $F_\even(k,c)=0$; then $\cos(k\pi
c)=-i\cos(k\pi)-2\sin(k\pi)$ and therefore, combining \eqref{f1} and
\eqref{f2}, we obtain
\[
\rme^{y\pi} \leq \rme^{y\pi|c |}+1+3\rme^{-y\pi}.
\]
If $y\geq \log(3)/\pi$ or equivalently $\rme^{y\pi}\geq 3$ then
\[
\rme^{y\pi}\leq \rme^{y\pi|c |}+2\leq \rme^{y\pi|c
|}+\frac{2}{3}\rme^{y\pi}.
\]
A simple manipulation then yields that
$y\leq\frac{\log(3)}{\pi(1-\abs{c})}$, and the required result
follows.
\end{proof}
It follows that for $c\in[0,1)$ the even resonances are located in
the rectangles
$$
\Pi^\even_n(c) = \left\{x+iy: \abs{x-n}<\tfrac12,\ 0\leq y\leq
\tfrac{\log 3}{\pi(1-\abs{c})}\right\}.
$$
Just as in the odd case, the following statement shows that for each
$n\in\Z$ and $c\in[0,1)$, the rectangle $\Pi^\even_n(c)$ contains
exactly one resonance.

\begin{theorem}\label{cequals02}
If $c =0$ there is an even resonance of the algebraic multiplicity
one at $k=n-i\log(3)/\pi$ for every even $n\in\Z$ and an even
eigenvalue of multiplicity one at $k=n$ for every non-zero odd
$n\in\Z$. There are no other even resonances.
\end{theorem}
The proof follows from the explicit formula
$$
F_\even(k,0)=-\frac{i}2(e^{ik\pi} -3)(1+e^{-ik\pi}).
$$
Just as in the odd case, we obtain that the resonances are given by
branches of real analytic functions of $c\in[0,1)$ with values in
$\Pi^\even_n(c)$. However, in contrast with the odd case, the height
of the rectangles $\Pi^\even_n(c)$ is not uniformly bounded in $c$.
Moreover, we have

\begin{theorem}
Let $n\in\Z$ and let $k_n=k_n(c)$ be the unique solution to
$F_\even(k,c)=0$ with $k_n(c)\in\Pi^\even_n(c)$. Then $\Im
k_n(c)\to-\infty$ as $c\to1$.
\end{theorem}
\begin{proof}
Suppose that the conclusion of the theorem is false. Then there
exists a sequence $c_m\to1$ such that $\Im k_n(c_m)$ is bounded. By
passing to a subsequence we can assume that $k_n(c_m)\to k_n^\infty
\in\C$ as $m\to \infty$. This would imply that
$F_\even(k_n^\infty,1)=0$ by the joint continuity of the function
$F_\even$. This is impossible by \eqref{d13}.
\end{proof}
Formal calculations and numerical analysis suggest that the rate of
divergence of $\Im k_n(c)$ as $c\to1$ is logarithmic. Thus, all even
resonances move off to infinity and this provides partial
explanation for the failure of the Weyl law for $c=1$.

As in the odd case, the even resonance curves intersect the real
axis for some rational values of $k$. A direct computation shows
that $F_\even(k,c)=0$ for $k\in\R$ if and only if
$$
k=m+n-1,\quad  c =\frac{m-n}{m+n-1}, \text{ for some $n,m\in\N$.}
$$

Figure~\ref{figure2} plots a typical even resonance curve as $c$ increases from $0$ to $1$. The curve starts at $4-i\log(3)/\pi$ when $c=0$ and passes through $4$ when $c=\frac{1}{5},\, \frac{3}{5}$ and diverges to $\infty$ as $c\to 1$.

\begin{figure}[h!]
\begin{center}
\hspace*{-2em}\includegraphics[width=6in]{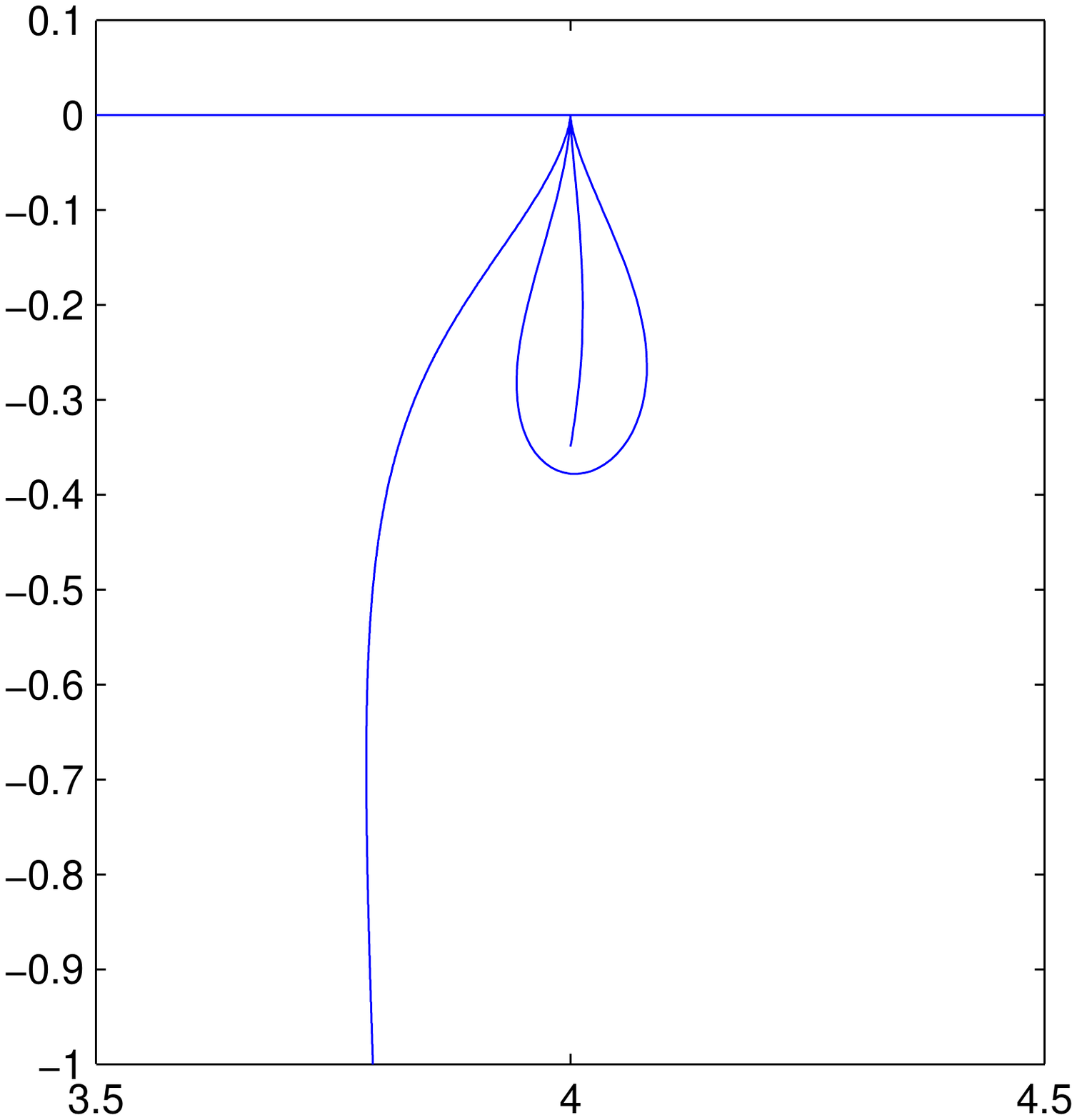}
\end{center}
\vspace*{-4ex} \caption{The even resonance curve in $\Pi^\even_4$}
\label{figure2}
\end{figure}

\end{document}